\documentclass[nomath,reqno,a4paper,12pt]{amsart}
\makeatletter \let\th@plain\relax \makeatother 

\usepackage{mystyle}

\usepackage[utf8]{inputenc}
\usepackage[T1]{fontenc}
\usepackage{lmodern}
\usepackage[british]{babel}
\usepackage{microtype}
\usepackage[headings]{fullpage}
\usepackage[skip=0.5\baselineskip plus 2.05pt]{parskip}

\usepackage[hyphens]{url}%
\usepackage[pdfencoding=auto,pdfusetitle]{hyperref}
\hypersetup{colorlinks=true, linkcolor=blue!30!black, citecolor=green!30!black, urlcolor=red!45!black}
\usepackage[hyphenbreaks]{breakurl}%

\usepackage{microtype}
\usepackage{booktabs}

\usepackage{amsfonts,amssymb,bbm,stmaryrd}
\usepackage[scr=boondoxo,scrscaled=1.05]{mathalfa}
\usepackage{amsmath}
\usepackage[thmmarks, amsmath, amsthm]{ntheorem}
\theorempreskip{\parsep} 
\AtBeginEnvironment{proof}{\vspace{\parsep}} 

\usepackage{mathtools}
\usepackage[noabbrev,capitalise]{cleveref}
\usepackage{centernot,float,subcaption,array,booktabs,colortbl}%

\usepackage[shortlabels]{enumitem}
\setlist[enumerate,1]{label={\roman*)}}

\usepackage{tikz}
\usetikzlibrary{matrix,positioning}
\usetikzlibrary{calc}
\usetikzlibrary{arrows}
\tikzset{> =stealth}

\theoremstyle{plain}
\newtheorem{theorem}{Theorem}[section]
\newtheorem{lemma}[theorem]{Lemma}
\newtheorem{proposition}[theorem]{Proposition}
\newtheorem{corollary}[theorem]{Corollary}

\theoremstyle{definition}
\newtheorem{definition}[theorem]{Definition}

\theoremstyle{remark}
\newtheorem{remark}[theorem]{Remark}

\renewcommand{\epsilon}{\varepsilon}
\renewcommand{\phi}{\varphi}

\DeclareFontEncoding{MDA}{}{}
\DeclareSymbolFont{mathdesignA}{MDA}{mdput}{m}{n}%
\DeclareMathSymbol{\digamma}{\mathord}{mathdesignA}{"62}

\mathchardef\mhyphen="2D

\renewcommand{\land}{\mathrel{\wedge}}
\renewcommand{\lor}{\mathrel{\vee}}

\newcommand{\Q}{\mathbb{Q}}
\newcommand{\R}{\mathbb{R}}

\newcommand{\E}{\mathcal{E}}
\newcommand{\F}{\mathcal{F}}
\def\U{\mathcal{U}} 
\newcommand{\V}{\mathcal{V}}

\newcommand{\Sloc}{\mathcal{S}}

\newcommand{\id}{\mathrm{id}}

\newcommand{\Frm}{\mathbf{Frm}}
\newcommand{\Loc}{\mathbf{Loc}}

\newcommand{\OLoc}{\mathbf{OLoc}}

\newcommand{\PreUnifLoc}{\mathbf{PUnifLoc}}
\newcommand{\UnifLoc}{\mathbf{UnifLoc}}
\newcommand{\CUnifLoc}{\mathbf{CUnifLoc}}
\newcommand{\PreMetLoc}{\mathbf{PMetLoc}}
\newcommand{\MetLoc}{\mathbf{MetLoc}}
\newcommand{\CMetLoc}{\mathbf{CMetLoc}}
\newcommand{\OLocGrp}{\mathbf{OLocGrp}}

\newcommand{\op}{{^\mathrm{\hspace{0.5pt}op}}}

\renewcommand{\O}{\mathcal{O}}

\newcommand{\st}{\mathrm{st}}
\newcommand{\wkclo}[1]{\mathrm{wk}\mhyphen \mathrm{cl}({#1})}
\newcommand{\Cvar}{\mathcal{C}}
\newcommand{\Cauchy}{\mathscr{C}}

\DeclarePairedDelimiter\abs{\lvert}{\rvert}
\DeclarePairedDelimiter\norm{\lVert}{\rVert}
\makeatletter
\let\oldabs\abs
\def\abs{\@ifstar{\oldabs}{\oldabs*}}
\let\oldnorm\norm
\def\norm{\@ifstar{\oldnorm}{\oldnorm*}}
\makeatother

\makeatletter
\let\save@mathaccent\mathaccent
\newcommand*\if@single[3]{%
  \setbox0\hbox{${\mathaccent"0362{#1}}^H$}%
  \setbox2\hbox{${\mathaccent"0362{\kern0pt#1}}^H$}%
  \ifdim\ht0=\ht2 #3\else #2\fi
  }
\newcommand*\rel@kern[1]{\kern#1\dimexpr\macc@kerna}
\newcommand*\wideaccent[2]{\@ifnextchar^{{\wide@accent{#1}{#2}{0}}}{\wide@accent{#1}{#2}{1}}}
\newcommand*\wide@accent[3]{\if@single{#2}{\wide@accent@{#1}{#2}{#3}{1}}{\wide@accent@{#1}{#2}{#3}{2}}}
\newcommand*\wide@accent@[4]{%
  \begingroup
  \def\mathaccent##1##2{%
    \let\mathaccent\save@mathaccent
    \if#42 \let\macc@nucleus\first@char \fi
    \setbox\z@\hbox{$\macc@style{\macc@nucleus}_{}$}%
    \setbox\tw@\hbox{$\macc@style{\macc@nucleus}{}_{}$}%
    \dimen@\wd\tw@
    \advance\dimen@-\wd\z@
    \divide\dimen@ 3
    \@tempdima\wd\tw@
    \advance\@tempdima-\scriptspace
    \divide\@tempdima 10
    \advance\dimen@-\@tempdima
    \ifdim\dimen@>\z@ \dimen@0pt\fi
    \rel@kern{0.6}\kern-\dimen@
    \if#41
      #1{\rel@kern{-0.6}\kern\dimen@\macc@nucleus\rel@kern{0.4}\kern\dimen@}%
      \advance\dimen@0.4\dimexpr\macc@kerna
      \let\final@kern#3%
      \ifdim\dimen@<\z@ \let\final@kern1\fi
      \if\final@kern1 \kern-\dimen@\fi
    \else
      #1{\rel@kern{-0.6}\kern\dimen@#2}%
    \fi
  }%
  \macc@depth\@ne
  \let\math@bgroup\@empty \let\math@egroup\macc@set@skewchar
  \mathsurround\z@ \frozen@everymath{\mathgroup\macc@group\relax}%
  \macc@set@skewchar\relax
  \let\mathaccentV\macc@nested@a
  \if#41
    \macc@nested@a\relax111{#2}%
  \else
    \def\gobble@till@marker##1\endmarker{}%
    \futurelet\first@char\gobble@till@marker#2\endmarker
    \ifcat\noexpand\first@char A\else
      \def\first@char{}%
    \fi
    \macc@nested@a\relax111{\first@char}%
  \fi
  \endgroup
}
\makeatother
\newcommand\linehat[1]{\overline{\widehat{#1}}}
\newcommand\widelinehat{\wideaccent\linehat}

\title{Uniform locales and their constructive aspects}

\author[G. Manuell]{Graham Manuell}
\address{CMUC, Department of Mathematics, University of Coimbra, 3000-143 Coimbra, Portugal}
\email{graham@manuell.me}

\date{January 2024}

\subjclass[2010]{54E15, 06D22, 03F65, 54B30}
\keywords{uniform frame, metric frame, open locale, localic algebra}

\begin{document}

\begin{abstract}
Much work has been done on generalising results about uniform spaces to the pointfree context. However, this has almost exclusively been done using classical logic, whereas much of the utility of the pointfree approach lies in its constructive theory, which can be interpreted in many different toposes.
Johnstone has advocated for the development of a constructive theory of uniform locales and wrote a short paper on the basic constructive theory via covering uniformities, but he never followed this up with a discussion of entourage uniformities or completions.

We present a more extensive constructive development of uniform locales, including both entourage and covering approaches, their equivalence, completions and some applications to metric locales and localic algebra.

Some aspects of our presentation might also be of interest even to classically minded pointfree topologists. These include the definition and manipulation of entourage uniformities using the internal logic of the geometric hyperdoctrine of open sublocales and the emphasis on pre-uniform locales. The latter leads to a description of the completion as the uniform reflection of the pre-uniform locale of Cauchy filters and a new result concerning the completion of pre-uniform localic rings, which can be used to easily lift addition and multiplication on $\Q$ to $\R$ (or $\Q_p$) in the pointfree setting.
\end{abstract}

\maketitle
\thispagestyle{empty}

\setcounter{section}{-1}
\section{Introduction}

Thirty years ago in \cite{johnstone1991constructive} Johnstone took the first steps in a constructive development of the theory of uniform locales. This was promised to be the first in a series of papers on the topic, but the later papers never materialised and no one else has taken up the mantle.

Still the classical theory of uniform locales (and uniform spaces) has continued unabated and it would be as useful as ever to have access to these techniques in the constructive setting, where the results we prove have greater applicability by being interpretable in any topos with natural numbers object.

This paper can be thought of as a somewhat belated sequel to \cite{johnstone1991constructive} (in addition to making some minor corrections and modifications to the theory developed there). We have nonetheless attempted to keep our development self-contained within reason, though the original paper still provides a useful source of counterexamples (which we do not discuss) and a more exploratory account of what the theory should be.

We describe uniform locales in terms of both entourages and covers and prove their equivalence before developing the theory of uniform completions. Our account is somewhat atypical in how strong a focus it places on pre-uniform locales. This provides an easy route to proving completeness and cocompleteness of the category $\UnifLoc$ of uniform locales by factoring the forgetful functor $\UnifLoc \to \Loc$ into a reflective subcategory $\UnifLoc \hookrightarrow \PreUnifLoc$, a topological functor $\PreUnifLoc \to \OLoc$ and coreflective subcategory $\OLoc \hookrightarrow \Loc$. It also allows us to phrase the relationship between the classifying locale of all Cauchy filters and the completion in terms of the uniform reflection.

However, the true motivation for this approach is revealed in the final section where we discuss a few important examples of uniform locales and make some first applications to localic algebra. Here it seems to be important to be able to take completions of pre-uniform locales, since the uniform reflection might not preserve products, while its composite with the completion always does. This is apparent already in the example of completing the rational numbers to give the localic ring of real numbers.

We also provide a constructive proof of the classically known result that an overt localic group is complete with respect to its two-sided uniformity.
Further properties of uniform locales, such as total boundedness, will be left for a later paper.

Let me end this section by drawing the reader's attention to some other related work in formal topology that I became aware of after starting to write this paper.
In \cite{kawai2017completion}, Kawai describes the completion a uniform \emph{space} as formal topology. He uses the gauge approach to uniform spaces, and as I understand it, what are called (generalised) uniform spaces there are closer to what are sometimes called (quasi-)gauge spaces and have different morphisms.
The paper \cite{curi2006collection} also considers a notion of uniform formal topologies using gauges, but the notion of completeness considered there depends on the (global) points of the formal topology and thus appears to be weaker than the usual definition of completeness for uniform locales.

Finally, Kawai mentions in \cite{kawai2017completion} that the PhD thesis of Fox \cite{fox2005thesis} also discusses uniform formal topologies, though unfortunately I was unable to obtain a copy of this thesis until very recently. The thesis is worth reading. Fox defines uniform formal topologies in terms of covers, describes their relation to metric formal topologies and gauges and defines a uniform completion.
Thus, despite not discussing entourages there is some overlap with the current paper. Nonetheless, our approach is significantly different and I believe that the locale-theoretic framework used in this paper will be more easily understandable to a wider audience than the approach via formal topology used by Fox.

\section{Background}

We assume a basic knowledge of (at least the classical theory of) frames and locales. See \cite{picado2012book} for a good introduction,
which also has a nice account of the classical situation regarding uniform frames and locales.
Background on uniform spaces is not strictly necessary, but can be found in standard general topology textbooks such as \cite{willard1970general}.

We denote the category of locales by $\Loc$ and the category of frames by $\Frm$. We will usually start with a locale $X$ and write $\O X$ for its associated frame.
We maintain a strict notational distinction between these to avoid confusion between an open $a \in \O X$ and an element of $x$ of $X$ in the internal logic (described below).
If $f\colon X \to Y$ is a locale morphism, we write $f^*$ for the corresponding frame map and $f_*$ for its right adjoint. If $f^*$ has a left adjoint, we will write this as $f_!$.

We do \emph{not} assume the reader has any great familiarity with constructive pointfree topology and will attempt to provide a brief introduction below. A reasonably extensive introduction to the topic can be found in the background chapter of \cite{ManuellThesis} and further information can be found in \cite[Part C]{Elephant2}. The constructive theory of metric locales is described in \cite{Henry2016}.

Some understanding of category theory is also assumed. We make some use of the theory of topological functors, a good account of which can be found in \cite{JoyOfCats}.
We will also use the internal logic of a geometric hyperdoctrine, the basic idea of which we describe below. A slightly longer overview can again be found in the background of \cite{ManuellThesis},
but for a more careful and thorough introduction see \cite{PittsLogic}.

\subsection{Positivity, overtness and hyperdoctrines}

Constructively, the initial frame is the lattice of truth values $\Omega$, though this is no longer necessarily isomorphic to the two-element set $\{\top,\bot\}$.
When using constructive logic, we do not have access to double negation elimination and so it is best to phrase definitions in a `positive' way (without using negations) since negations can be difficult to get rid of once they are introduced. An example of this is can be seen in the relevant importance of nonempty versus inhabited sets. At set $X$ is \emph{inhabited} if $\exists x \in X$. This is more useful than the weaker condition that $X \ne \emptyset$, since we have access to the element $x$ to use in later arguments.

A similar concept appears in pointfree topology. An element $a$ of a frame $\O X$ is said to be \emph{positive} (written $a > 0$) if whenever $a \le \bigvee A$ then the set $A$ is inhabited.
Classically, this is equivalent requiring $a \ne 0$, but constructively it is a stronger condition. We say the locale $X$ is positive if $1 > 0$ in $\O X$.
If $V$ is a sublocale of $X$ we write $V \between a$ to mean that the restriction of $a$ to $V$ is positive. In particular, $a \between b \iff a \wedge b > 0$.

A locale $X$ is said to be \emph{overt} if it has a base of positive elements. Some people also call such locales \emph{locally positive} or even \emph{open}. Classically, every locale is overt. Constructively, the condition is nontrivial, though it still holds for many locales that appear in practice. A locale $X$ is overt if and only if the unique map $!\colon X \to 1$ is open. In this case, the left adjoint $\exists\colon \O X \to \Omega$ of the associated frame map measures the positivity of elements: $\exists(a) = \top \iff a > 0$. Another characterisation gives that $X$ is overt if the product projection $\pi_2\colon X \times Y \to Y$ is open for every locale $Y$. In this way it can be seen to be a kind of `dual' to compactness. Finally, we note that open sublocales of overt locales are overt.

We write $\OLoc$ for the category of overt locales and locale morphisms. This is a coreflective subcategory of $\Loc$. (A proof of this fact can be found in \cite{maietti2004structural}, but the main idea is that every locale has a largest overt sublocale.) Thus, $\OLoc$ is complete and cocomplete. Moreover, it can be shown that $\OLoc$ is closed under finite products in $\Loc$
and that $\OLoc$ has (epi, extremal mono)-factorisations where the extremal monomorphisms are precisely the overt sublocale inclusions.

The obvious functor $\O\colon \OLoc\op \to \Frm$ sending an overt locale to its frame of opens satisfies the necessary axioms to be a \emph{geometric hyperdoctrine (without equality)}. This means that for every product projection $\pi_2 \colon X \times Y \to Y$, the map $\O(\pi_2) = \pi_2^*$ has a left adjoint $\exists^X$ satisfying the Frobenius reciprocity condition $\exists^X(a \wedge \pi_2^*(b)) = \exists^X(a) \wedge b$, and moreover, the family of maps $(\exists^X)_Y \colon \O(X \times Y) \to \O Y$ is natural in $Y$ (the Beck--Chevalley condition).
These hold because the projections are open maps, which is why we need overtness.

This geometric hyperdoctrine allow us to interpret \emph{geometric logic} in the category $\OLoc$, with predicates being given by opens. We can discuss \emph{equality judgements} $t =_{\vec{x}} t'$, which state that terms (given by morphisms in the category) are equal, and \emph{sequents} $\phi \vdash_{\vec{x}} \psi$ where $\phi$ and $\psi$ are formulae in the variables $\vec{x}$ involving finite conjunctions, possibly infinitary disjunctions and existential quantification, which mean that the open defined by $\phi$ is contained in that defined by $\psi$. In particular, this logic contains regular logic as fragment and so we have a well-behaved calculus of open relations.
For the details behind the interpretation of logic in hyperdoctrines see \cite{PittsLogic}, while more information about the specific case of $\O\colon \OLoc\op \to \Frm$ can be found in \cite{ManuellThesis}.

\subsection{Strong density and weak closedness}

Another important concept we will need is that of strong density. A locale map $f\colon X \to Y$ is \emph{dense} if $f^*(a) = 0 \implies a = 0$. If $X$ and $Y$ are overt locales, $f$ is \emph{strongly dense} (or \emph{fibrewise dense}) if $a > 0 \implies f^*(a) > 0$. Strong density implies density, but the converse cannot be proved constructively. It is strong density that is more important for the study of uniform locales.

We can express the condition for strong density in terms of $\exists$ as $\exists(a) \le \exists f^* (a)$. Then taking right adjoints gives the equivalent condition that $f_*{!}^*(p) \le {!}^*(p)$ for all $p \in \Omega$. (This definition even makes sense when $X$ and $Y$ are not overt, though we will not need this.)

Strongly dense maps form a factorisation system on $\OLoc$ together with the \emph{weakly closed sublocales}. In particular, every overt sublocale $V \hookrightarrow X$ has a \emph{weak closure} $\wkclo{V} \hookrightarrow X$,
which is the largest sublocale in which it is strongly dense (and this is necessarily overt). It can be shown that if $V$ is an overt sublocale then $V \between a \iff \wkclo{V} \between a$.
Every closed sublocale is weakly closed, but all sublocales of discrete locales are weakly closed, while all open sublocales of $1$ are closed only if excluded middle holds.

With this weaker notion of closedness come weakened separation axioms. In particular, we say a locale $X$ is \emph{weakly Hausdorff} if the diagonal is \emph{weakly} closed.
Defining $a$ to be \emph{weakly rather below} $b$ if $\wkclo{a} \le b$ in the order of sublocales we also obtain a notion of \emph{weak regularity}. See \cite{johnstone1990fibrewise} for more details.

\subsection{Other miscellaneous results}

As is well understood, the lattice of sublocales of a locale $X$ is a coframe $\Sloc X$. Moreover, this is functorial so that for $f\colon X \to Y$ we have a coframe homomorphism $\Sloc f\colon \Sloc Y \to \Sloc X$, which is obtained by pulling back along $f$. This map then has a left adjoint $(\Sloc f)_!\colon \Sloc X \to \Sloc Y$ which can be understood as taking images. Now if $V$ is an overt sublocale of $X$ it can be shown that $(\Sloc f)_!(V)$ is overt and $(\Sloc f)_!(V) \between a \iff V \between f^*(a)$.

We end this section by mentioning some results concerning products which will be of use to us.

\begin{lemma}
 Let $f\colon X \to Y$ and $g\colon X' \to Y'$ be strongly dense maps between overt locales. Then their product $f \times g\colon X \times X' \to Y \times Y'$ is strongly dense.
\end{lemma}
\begin{proof}
 Consider the following pullback diagrams.
 \begin{center}
 \begin{minipage}{0.4\textwidth}
  \centering
  \begin{tikzpicture}[node distance=2.75cm, auto]
    \node (A) {$X \times X'$};
    \node (B) [below of=A] {$X$};
    \node (C) [right=1.8cm of A] {$Y \times X'$};
    \node (D) [below of=C] {$Y$};
    \draw[->] (A) to node [swap] {$\pi_1$} (B);
    \draw[->] (A) to node {$f \times X'$} (C);
    \draw[->] (B) to node [swap] {$f$} (D);
    \draw[->] (C) to node {$\pi_1$} (D);
    \begin{scope}[shift=({A})]
        \draw +(0.25,-0.75) -- +(0.75,-0.75) -- +(0.75,-0.25);
    \end{scope}
  \end{tikzpicture}
 \end{minipage}
 \begin{minipage}{0.4\textwidth}
  \centering
  \begin{tikzpicture}[node distance=2.75cm, auto]
    \node (A) {$Y \times X'$};
    \node (B) [below of=A] {$X'$};
    \node (C) [right=1.8cm of A] {$Y \times Y'$};
    \node (D) [below of=C] {$Y'$};
    \draw[->] (A) to node [swap] {$\pi_2$} (B);
    \draw[->] (A) to node {$Y \times g$} (C);
    \draw[->] (B) to node [swap] {$g$} (D);
    \draw[->] (C) to node {$\pi_2$} (D);
    \begin{scope}[shift=({A})]
        \draw +(0.25,-0.75) -- +(0.75,-0.75) -- +(0.75,-0.25);
    \end{scope}
  \end{tikzpicture}
 \end{minipage}
 \end{center}
 Since $X'$ and $Y$ are overt, the product projections $\pi_1$ and $\pi_2$ are open. But strongly dense maps are stable under pullback along open maps (see \cite[Lemma 1.9]{johnstone1989constructive}) and hence $f \times X'$ and $Y \times g$ are strongly dense. Thus, so is the composite $f \times g = (Y \times g)\circ(f \times X')$.
\end{proof}

\begin{lemma}\label{lem:adjoint_of_product_map}
 Let $f\colon X \to Y$ and $g\colon X' \to Y'$ be locale morphisms. Then $(f \times g)_*(c) = \bigvee\{f_*(a) \oplus g_*(b) \mid a \oplus b \le c\}$.
\end{lemma}
\begin{proof}
 See \cite[Section 3, Lemma 2]{banaschewski1999completeness} (though the proof is not difficult).
\end{proof}

\section{Entourage uniformities}\label{sec:entourages}

Uniform locales are probably most easily understood via the entourage approach. An (open) entourage on a locale $X$ can be thought of an open approximate equality relation on $X$.

In order to have a good theory of open relations, we require a well-behaved notion of existential quantification for open subobjects and we will therefore restrict our attention to overt locales.
Since $\O\colon \OLoc\op \to \Frm$ is in particular a regular hyperdoctrine,
we may define composition of open relations $E,F \in \O(X \times X)$ in the internal logic by
\[F \circ E = \{(x,z) \colon X \times X \mid \exists y\colon X.\ (x,y) \in E \land (y,z) \in F\}.\]

We can now use this to mimic the usual definition of a uniform space.
\begin{definition}
 A \emph{pre-uniform locale} is an overt locale $X$ equipped with a filter $\E$ on $\O(X \times X)$ such that for each $E \in \E$,
 \begin{itemize}
  \item $\vdash_{x\colon X} (x,x) \in E$,
  \item there is an $E^\mathrm{o} \in \E$ such that $(x,y) \in E \dashv\vdash_{x, y\colon X} (y,x) \in E^\mathrm{o}$,
  \item there is an $F \in \E$ such that $F \circ F \le E$.
 \end{itemize}
 We call $\E$ a \emph{uniformity} and the elements of $\E$ \emph{entourages}. We say a set $\mathcal{B} \subseteq \E$ is a \emph{base} for the uniformity $\E$ if $\E$ is generated by $\mathcal{B}$ as an upset.
\end{definition}

This is not yet what is usually called a \emph{uniform} locale, since we have not assumed a compatibility condition between the uniformity $\E$ and the `topology' of $X$ and so it is possible for $X$ to have a finer topology than that induced by $\E$. In particular, $X$ could be a discrete locale in which case we recover the definition of a uniform space (but where we do not think of $X$ as being equipped with the usual uniform topology).
\begin{definition}
 We define the \emph{uniformly below relation} on the opens of a pre-uniform locale $(X,\E)$ by \[a\vartriangleleft b \iff \exists E \in \E.\ E \circ (a \oplus a) \le b \oplus b.\]
 Then a \emph{uniform locale} is a pre-uniform locale $(X, \E)$ such that every $b \in \O X$ can be expressed as \[b = \bigvee_{a \vartriangleleft b} a.\]
 In this case the uniformity $\E$ is said to be \emph{admissible}.
\end{definition}

The following lemma provides an intuitive way to understand the uniformly below relation using the internal logic.
\begin{lemma}\label{lem:uniformly_below_internal_logic}
 Let $E$ be an entourage. We have $E \circ (a \oplus a) \le b \oplus b$ if and only if \[y \in a \land (y,z) \in E \vdash_{y, z\colon X} z \in b\] in the internal logic.
 An equivalent expression is $\exists y\colon X.\ y \in a \land (y,z) \in E \vdash_{z\colon X} z \in b$.
\end{lemma}
\begin{proof}
 Let us start with the forward direction. In the internal logic the assumption means $\exists y\colon X.\ x \in a \land y \in a \land (y,z) \in E \vdash_{x, z\colon X} x \in b \land z \in b$. We can simply ignore the $x \in b$ part of the consequent and eliminate the existential quantification to obtain $x \in a \land y \in a \land (y,z) \in E \vdash_{x, y, z\colon X} z \in b$.
 But now substituting in $y$ for $x$ yields the desired sequent.
 
 For the converse, note that the assumption implies $x \in a \vdash_{x\colon X} x \in b$ by reflexivity (taking $z = y = x$). Then this together with the assumption again easily gives the desired statement.
 Finally, the second form easily seen to be equivalent to the first by the rules for existential quantification.
\end{proof}

The uniformly below relation satisfies many of the axioms of a \emph{strong inclusion} on $\O X$ (see \cite{banaschewski1990compactification}).
\begin{lemma}\label{lem:uniformly_below_properties}
 Let $(X, \E)$ be a pre-uniform locale. Then the uniformly below relation satisfies:
 \begin{enumerate}
  \item ${\le} \circ {\vartriangleleft} \circ {\le} \subseteq {\vartriangleleft} \subseteq {\le}$,
  \item ${\vartriangleleft}$ is a sublattice of $\O X \times \O X$,
  \item ${\vartriangleleft}$ is interpolative --- that is, ${\vartriangleleft} \subseteq {\vartriangleleft} \circ {\vartriangleleft}$.
 \end{enumerate}
\end{lemma}
\begin{proof}
 i) For the first inclusion, suppose $a' \le a \vartriangleleft b \le b'$. Then $E \circ (a \oplus a) \le b \oplus b$ for some $E \in \E$,
 and hence we easily find $E \circ (a' \oplus a') \le E \circ (a \oplus a) \le b \oplus b \le b' \oplus b'$, so that $a' \vartriangleleft b'$,
 as required.
 
 We actually already proved the second inclusion in \cref{lem:uniformly_below_internal_logic} using reflexivity and the equivalent description of the uniformly below relation given there.
 
 ii) It is easy to see that $0 \vartriangleleft 0$ and $1 \vartriangleleft 1$. It then suffices to show $a \vartriangleleft b$ and $a \vartriangleleft b'$ implies $a \vartriangleleft b \wedge b'$ and that $a \vartriangleleft b$ and $a'\vartriangleleft b$ implies $a \vee a' \vartriangleleft b$.
 
 For the former, we have $E \circ (a \oplus a) \le b \oplus b$ and $E' \circ (a \oplus a) \le b' \oplus b'$ for some $E, E' \in \E$.
 Thus, $(E \wedge E') \circ (a \oplus a) \le ( E \circ (a \oplus a) ) \wedge (E' \circ (a \oplus a)) \le (b \wedge b') \oplus (b \wedge b')$.
 Since $E \wedge E' \in \E$, we may conclude that $a \vartriangleleft b \wedge b'$.
 
 For the latter, we have $E \circ (a \oplus a) \le b \oplus b$ and $E' \circ (a' \oplus a') \le b \oplus b$ for some $E, E' \in \E$. By taking the meet we may assume $E = E'$ without loss of generality.
 Thus, by \cref{lem:uniformly_below_internal_logic} we have $y \in a \land (y,z) \in E \vdash_{y,z\colon X} z \in b$ and $y \in a' \land (y,z) \in E \vdash_{y,z\colon X} z \in b$
 and so we may conclude $(y \in a \land (y,z) \in E) \lor (y \in a' \land (y,z) \in E) \vdash_{y,z\colon X} z \in b$. The equivalent form of $E \circ ((a \vee a') \oplus (a \vee a')) \le b \oplus b$ then follows by distributivity, as required.
 
 iii) Suppose $a \vartriangleleft b$. Then $E \circ (a \oplus a) \le b \oplus b$ for some $E \in \E$. Take $F \in \E$ such that $F \circ F \le E$ and set
      $c = \{ z\colon X \mid \exists y\colon X.\ y \in a \land (y,z) \in F \}$. Again by the characterisation in \cref{lem:uniformly_below_internal_logic}, we have $a \vartriangleleft c$.
      Moreover, $F \circ (c \oplus c) \le b \oplus b$ if and only if $\exists y\colon X.\ \exists x\colon X.\ x \in a \land (x,y) \in F \land (y,z) \in F \vdash_{z\colon X} z \in b$.
      This is equivalent to $\exists x\colon X.\ x \in a \land (x,z) \in F \circ F \vdash_{z\colon X} z \in b$, which holds since $F \circ F \circ (a \oplus a) \le E \circ (a \oplus a) \le b \oplus b$.
      Thus, $c \vartriangleleft b$ and $\vartriangleleft$ interpolates.
\end{proof}

Morphisms of uniform locales can be defined straightforwardly.
\begin{definition}
 A \emph{morphism of (pre-)uniform locales} $f\colon (X, \E) \to (Y,\F)$ is a morphism of locales $f\colon X \to Y$ such that $(f \times f)^*(F) \in \E$ for all $F \in \F$.
\end{definition}

\begin{lemma}
 If $f\colon (X, \E) \to (Y,\F)$ is a morphism of pre-uniform locales, then $f^*$ preserves the uniformly below relation.
\end{lemma}
\begin{proof}
 Suppose $a \vartriangleleft b$ in $Y$. Then there is an $F \in \F$ such that $F \circ (a \oplus a) \le b \oplus b$. Applying $f^* \oplus f^*$ we have $(f^* \oplus f^*)(F \circ (a \oplus a)) \le f^*(b) \oplus f^*(b)$.
 Note that in the internal logic $(x, z) \in (f^* \oplus f^*)(F \circ (a \oplus a))$ holds if and only if $\exists y\colon Y.\ (f(x), y) \in a \oplus a \land (y,f(z)) \in F$.
 This is implied by $\exists y'\colon X.\ (f(x), f(y')) \in a \oplus a \land (f(y'),f(z)) \in F$ and hence $(f^* \oplus f^*)(F) \circ (f^*(a) \oplus f^*(a)) \le (f^* \oplus f^*)(F \circ (a \oplus a)) \le f^*(b) \oplus f^*(b)$. Therefore, since $(f^* \oplus f^*)(F) \in \E$ we have $f^*(a) \vartriangleleft f^*(b)$.
\end{proof}

\begin{proposition}
 The category $\PreUnifLoc$ of pre-uniform locales is topological over $\OLoc$.
\end{proposition}
\begin{proof}
 Consider a family of pre-uniform locales $(X_i, \E_i)$ and locale morphisms $f_i\colon Y \to X_i$ where $i \in I$. We claim that there is an initial lift $(Y, \E)$.
 Here we define $\E$ to be the filter generated by $(f_i \times f_i)^*(E_i)$ for each $E_i \in \E_i$ and each $i \in I$.
 
 Note that this is indeed a uniformity. First observe that $(f_i \times f_i)^*(E_i)$ always satisfies the reflexivity condition and that this condition is stable under finite meets. It is also easy to see that the symmetry condition is inherited from the $E_i$ uniformities. Finally, consider a general entourage $E = \bigwedge_{j \in J} (f_j \times f_j)^*(E_j) \in \E$ where $J$ is a finite subset of $I$.
 For each $E_j$ there is a $F_j \in \E_j$ such that $F_j \circ F_j \le E_j$. Then as we saw before, $(f_j \times f_j)^*(F_j) \circ (f_j \times f_j)^*(F_j) \le (f_j \times f_j)^*(F_j \circ F_j) \le (f_j \times f_j)^*(E_j)$,
 and therefore setting $F = \bigwedge_{j \in J} (f_j \times f_j)^*(F_j) \in \E$ we have $F \circ F \le E$. Thus, the final axiom of an entourage is satisfied.
 
 It is now routine to check that $(Y, \E)$ satisfies the necessary property in order to be a universal lift.
\end{proof}

By a general result, limits in $\PreUnifLoc$ can be computed as the limits of the underlying overt locales equipped with the initial uniformity with respect to the limiting cone.
Moreover, the existence of initial structures implies the existence of final structures.
Then colimits can be computed dually and we obtain the following corollary.

\begin{corollary}
 The category $\PreUnifLoc$ is complete and cocomplete.
\end{corollary}

The following result allows the category $\UnifLoc$ of uniform locales to inherit some of the good behaviour of the category of pre-uniform locales.
We state it here for convenience, but since it will be easier to prove with covering uniformities, we postpone the proof until \cref{sec:covering_uniformities}.

\begin{proposition}\label{prop:uniform_reflection}
 Uniform locales form a reflective subcategory of the category of pre-uniform locales.
\end{proposition}

\begin{corollary}
 The category $\UnifLoc$ is complete and cocomplete.
\end{corollary}

We can also use this approach to define the notion of a uniform embedding.
\begin{definition}
 A morphism of (pre-)uniform locales is a \emph{uniform embedding} if it is an initial locale embedding --- that is, if it is initial with respect to the topological functor $\PreUnifLoc \to \OLoc$ and its underlying locale morphism is an extremal monomorphism.
 By general principles, these maps are precisely the extremal monomorphisms in $\PreUnifLoc$.
\end{definition}

Explicitly, we have that if $(X,\E)$ is a pre-uniform locale and $i\colon S \hookrightarrow X$ is an overt sublocale, then $i$ is a uniform embedding if and only if the $S$ is equipped with the uniformity $\{ (i \times i)^*(E) \mid E \in \E\}$. Moreover, it is easy to see that in this case $S$ is a uniform locale whenever $X$ is.

\section{Covering uniformities}\label{sec:covering_uniformities}

In pointfree topology it is more common to describe uniform locales via uniform covers and this is the approach taken by Johnstone in \cite{johnstone1991constructive}.
However, as pointed out in \cite{manuell2020congruence}, all definitions of uniform spaces or locales via covers that I have been able to find \cite{tukey1941convergence,isbell1964uniform,willard1970general,frith1987thesis,johnstone1991constructive,picado2012book} are incorrect. Classically the error is a minor one, since it only manifests for uniformities on the empty set (or the initial locale), but constructively it is more important. I explain this error after the definitions below.

\begin{definition}
 A \emph{covering downset} $C$ on a locale $X$ is a downset on $\O X$ such that $\bigvee C = 1$.
 We say a covering downset is \emph{strong} if it is generated by its positive elements.
\end{definition}

\begin{definition}
 Given a locale $X$, an overt sublocale $V$ of $X$ and a covering downset $U \subseteq \O X$, the \emph{star} of $V$ with respect to $U$ is defined by $\st(V,U) = \bigvee \{u \in U \mid V \between u\}$.
 In particular, this yields a notion of star of an open $a \in \O X$ when $X$ is overt.
 For a covering downset $U$ on an overt locale we set $U^\star = {\downarrow}\{ \st(u, U) \mid u \in U \}$.
\end{definition}

\begin{lemma}\label{lem:star_is_inflationary}
 Let $V$ and $U$ be as defined as above. Then $V \le \wkclo{V} \le \st(V,U)$ under the ordering of sublocales.
\end{lemma}
\begin{proof}
 First note that since $V \between a \iff \wkclo{V} \between a$, we have $\st(V, U) = \st(\wkclo{V},U)$
 and hence it is enough to prove $V \le \st(V, U)$.
 
 We write the frame quotient map associated to $V$ as $a \mapsto a \cap V$. To prove the result we will show $\st(V,U) \cap V = 1$ in $\O V$.
 
 Since $U$ is a cover, we have $\bigvee U = 1$ in $\O X$. Applying the quotient map then gives $\bigvee\{ u \cap V \mid u \in U\} = 1$ in $\O V$.
 Now since $V$ is overt, we may replace this join by a similar one consisting only of positive elements $\bigvee\{ u \cap V \mid u \in U, u \cap V > 0\}$.
 But $u \cap V > 0$ in $V$ if and only if $V \between u$ in $X$ and hence this join is in turn equal to $\left(\bigvee\{ u \in U \mid V \between u\}\right) \cap V$,
 which is simply $\st(V,U) \cap V$. Thus, we have shown $\st(V,U) \cap V = 1$, as required.
\end{proof}

\begin{definition}
 A \emph{pre-uniform locale via covers} is an overt locale $X$ equipped with a filter $\U$ of covering downsets such that for every $U \in \U$ there exists a \emph{strong} $V \in \U$ such that $V^\star \subseteq U$.
\end{definition}

\begin{remark}
 Where the usual definitions of covering uniformities go wrong is that they omit the strength condition on $V$. Without it, $\{\{0\}\}$ and the powerset of $\{0\}$ would be distinct valid uniformities on the trivial locale in disagreement with the entourage approach.
 Johnstone \cite{johnstone1991constructive} notices that this strength condition is important and calls a uniformity proper if it satisfies it, but does not require it for every uniformity, nor does he require (pre-)uniform locales to be overt in general. (Johnstone also fails to require that a uniformity be inhabited, but I imagine this was just an mistake in the writeup.)
\end{remark}

Nonetheless, it is easy to modify an otherwise valid covering uniformity to ensure that the strength condition holds.
It is not hard to see that a collection $\mathcal{B}$ of covering downsets that is a base for a filter which satisfies every property of a covering uniformity aside from the strength condition can be made into a base for a true uniformity by simply replacing each $B \in \mathcal{B}$ with ${\downarrow}\{u \in B \mid u > 0\}$.

\begin{definition}
 We define the uniformly below relation on the opens of a pre-uniform locale via covers $(X,\U)$ by \[a\vartriangleleft b \iff \exists U \in \U.\ \st(a,U) \le b.\]
 Then as before, a \emph{uniform locale via covers} is a pre-uniform locale via covers $(X, \U)$ such that every $b \in \O X$ can be expressed as $b = \bigvee_{a \vartriangleleft b} a$.
\end{definition}

As is well-known classically, the definition of uniformities via covers is equivalent to the entourage approach discussed in \cref{sec:entourages}.

\begin{definition}
 A \emph{morphism of (pre-)uniform locales via covers} $f\colon (X, \U) \to (Y,\V)$ is a morphism of locales $f\colon X \to Y$ such that ${\downarrow} f^*[V] \in \U$ for all $V \in \V$.
\end{definition}

\begin{theorem}\label{thm:entourage_vs_cover_preuniformity}
 There is an isomorphism of categories between the category of pre-uniform locales via entourages and the category of pre-uniform locales via covers (which commutes with the forgetful functor into $\OLoc$).
\end{theorem}
\begin{proof}
 Let $(X, \U)$ be a pre-uniform locale via covers. We define an entourage uniformity $\overline{\U}$ on $X$ with basic entourages of the form $\overline{U} = \bigvee\{u \oplus u \mid u \in U\}$ for each strong covering downset $U \in \U$.
 It is easy to see that $\overline{U \cap V} \le \overline{U} \wedge \overline{V}$ and so this is indeed a filter base. We now confirm that the three axioms hold.
 
 First observe that if $\Delta\colon X \to X \times X$ is the diagonal map, $\Delta^*(\overline{U}) = \bigvee\{u \wedge u \mid u \in U\} = 1$, since $U$ is a cover. But this is precisely what the reflexivity axiom is saying in the internal logic.
 
 Next we note that each of these basic entourages is fixed under the automorphism of $X \times X$ sending $(x,y)$ to $(y,x)$ in the internal logic. It follows easily that the symmetry condition holds for the resulting filter of entourages.
 
 Finally, we require for each basic entourage $\overline{U}$ an entourage $F \in \overline{\U}$ such that $F \circ F \le \overline{U}$.
 We claim that if $V$ is a strong covering downset such that $V^\star \subseteq U$, then $\overline{V}$ is such an $F$.
 Suppose $(x,z) \in \overline{V} \circ \overline{V}$ in the internal logic. This means there is a $y\colon X$ such that $(x,y) \in \overline{V}$ and $(y,z) \in \overline{V}$.
 But $\overline{V} = \bigvee\{v \oplus v \mid v \in V\}$ and so we may assume $x,y \in v$ and $y,z \in v'$ for some $v, v' \in V$.
 Thus, we have reduced the claim to the proving sequent $x \in v \land z \in v' \land \exists y\colon X.\ y \in v \wedge v' \vdash_{x,z\colon X} (x,z) \in \overline{U}$. But this sequent just means that $v \oplus v' \le \overline{U}$ whenever $v \between v'$. In that case, $v,v' \le \st(v,V)$. So $v \oplus v' \le \st(v,V) \oplus \st(v,V) \le \bigvee_{w \in V} \st(w,V) \oplus \st(w,V) = \overline{V^\star} \le \overline{U}$ and the claim follows.
 
 We now show that the map $(X, \U) \mapsto (X, \overline{\U})$ is functorial. We need only show that if $f^*\colon \O Y \to \O X$ sends uniform covers to uniform covers, then it sends entourages to entourages.
 It is enough to show this for basic entourages. Suppose $\overline{V}$ is a basic entourage on $Y$. We have $(f \times f)^*(\overline{V}) = \bigvee\{f^*(v) \oplus f^*(v) \mid v \in V\} = \bigvee\{u \oplus u \mid u \in {\downarrow}f^*[V]\} = \overline{{\downarrow}f^*[V]}$, which is a (basic) entourage on $X$.
 
 Now let $(X, \E)$ be a pre-uniform locale (via entourages). We define a covering uniformity $\widehat{\E}$ with basic covering downsets of the form $\widehat{E} = {\downarrow}\{u \in \O X \mid u > 0,\, u \oplus u \le E \}$ for each $E \in \E$. As above, reflexivity gives that $1 = \bigvee \{ u \wedge u' \mid u \oplus u' \le E \} \le \bigvee \{ u \mid u \oplus u \le E \}$. Then by overtness, we may restrict this join to the positive elements and so $\widehat{E}$ is indeed a cover. Again, it is easy to see that $\widehat{E \wedge E'} \le \widehat{E} \cap \widehat{E'}$ and so $\widehat{\E}$ is a filter base.
 To see that $\widehat{\E}$ is a covering uniformity, it just remains to show the `star-refinement' axiom.
 
 Let $\widehat{E}$ be a basic covering downset in $\widehat{\E}$. We must find a strong covering downset $V$ in $\widehat{\E}$ such that $V^\star \subseteq \widehat{E}$.
 By using the final `transitivity' axiom of entourage uniformities twice we have that there is an $F \in \E$ such that $F \circ F \circ F \le E$.
 We claim that $\widehat{F}$ is a $V$ satisfying the above condition.
 
 To see this, first note that $\widehat{F}$ is strong by construction. To show $\widehat{F}^\star \subseteq \widehat{E}$, it is enough to show that for all $v \in \widehat{F}$,
 $\st(v, \widehat{F}) \le u$ for some positive $u$ such that $u \oplus u \le E$. In particular, we can show $\st(v, \widehat{F}) \oplus \st(v, \widehat{F}) \le E$. (We may restrict to $v > 0$ since $\widehat{F}$ is strong and then $\st(v, \widehat{F}) \ge v > 0$.)
 Further expanding the definitions, we have $\st(v, \widehat{F}) = \bigvee\{ v' \in \O X \mid v' > 0,\ v' \oplus v' \le F ,\, v \between v' \}$ and so the desired inequality reduces to $v' \oplus v'' \le E$ for all $v',v''$ such that $v' \oplus v',\, v'' \oplus v'' \le F$, $v \between v'$ and $v \between v''$. We now use the internal logic. Take $x \in v'$ and $y \in v''$. There exist $x' \in v \wedge v'$ and $y' \in v \wedge v''$. Now since $x, x' \in v'$, we have $(x, x') \in F$. Similarly, $(y', y) \in F$. But recall that $v \in \widehat{F}$, so that $v \oplus v \le F$. Thus since $x', y' \in v$, we also have $(x', y') \in F$.
 So by the definition of relational composition, we have $(x,y) \in F \circ F \circ F \le E$. Therefore, $\widehat{F}^\star \subseteq \widehat{E}$ and $\widehat{\E}$ is a covering uniformity.
 
 As before, we show that $(X, \E) \mapsto (X, \widehat{\E})$ is functorial by proving that $f^*\colon \O Y \to \O X$ sends uniform covers to uniform covers whenever it sends entourages to entourages.
 Let $\widehat{F}$ be a basic uniform covering downset on $Y$.
 To show that this is sent to a uniform cover we require that $\widehat{E} \subseteq {\downarrow}f^*[\widehat{F}]$ for some entourage $E$ on $X$. We will take $E = (f \times f)^*(F')$ for an entourage $F'$ on $Y$ to be chosen later.
 So we require ${\downarrow}\{u \mid u > 0,\, u \oplus u \le (f \times f)^*(F') \} \subseteq {\downarrow}\{f^*(v) \mid v > 0,\, v \oplus v \le F \}$.
 Explicitly, for a positive $u$ such that $u \oplus u \le (f \times f)^*(F')$, we require a positive $v$ such that $v \oplus v \le F$ and $u \le f^*(v)$.
 To find such a $v$ we might try to take $v$ as small as possible such that $u \le f^*(v)$. If $f^*$ had a left adjoint $f_!$ we could achieve this by taking $v = f_!(u)$.
 In general this is not possible. However, we can approximate this by taking image of $u$ as a sublocale (which always exists) and then using the star operation to find an open containing this.
 In this way we set $v = \st((\Sloc f)_!(u), \widehat{F'})$. Then $f^*(v) = f^*(\st((\Sloc f)_!(u), \widehat{F'})) = \st(u, {\downarrow}f^*[\widehat{F'}]) \ge u$, as desired.
 Also note that $(\Sloc f)_!(u)$ is positive as the image of the positive open $u$ and hence $v > 0$.
 It remains to show $v \oplus v \le F$.
 
 This can be shown in a similar way to the proof of the star-refinement axiom above.
 Expanding the definitions, we must show that given $w, w'$ such that $w \oplus w \le F'$ and $w' \oplus w' \le F'$ and such that $u \between f^*(w)$ and $u \between f^*(w')$, we can conclude $w \oplus w' \le F$.
 In the internal logic, we take $x \in w$ and $y \in w'$. Then $\exists x'\colon X.\ x' \in u \land f(x') \in w$ and $\exists y'\colon X.\ y' \in u \land f(y') \in w'$. Recall $u \oplus u \le (f \times f)^*(F')$ and so we have $(x, f(x')) \in w \oplus w \le F'$, $(f(x'), f(y')) \in F'$ and $(f(y'), y) \in F'$.
 Thus, $(x,y) \in F' \circ F' \circ F'$. Now we can choose $F'$ to satisfy $F'\circ F'\circ F' \le F$ and so we are done.
 
 We have thus constructed functors from the category of pre-uniform locales via covers to the pre-uniform locales via entourages and back. Moreover, it is clear that these commute with the forgetful functor into $\OLoc$.
 It remains to show that these are inverses and it is enough to show this on objects.
 
 Consider a covering uniformity $\U$ on $X$. Then $\widehat{\overline{\U}}$ has a base consisting of covering downsets of the form
 $\widehat{\overline{U}} = {\downarrow}\{v \in \O X \mid v > 0,\, v \oplus v \le \bigvee\{u \oplus u \mid u \in U\}\}$ for each uniform covering downset $U \in \U$.
 We may restrict to strong $U$, in which case it is clear that $U \subseteq \widehat{\overline{U}}$ and hence $\widehat{\overline{\U}} \subseteq \U$.
 
 On the other hand, suppose $v$ is one of the generators of $\widehat{\overline{U}}$. Then $v > 0$ so that $v \between 1 = \bigvee U$ and hence $v \between u'$ for some $u' \in U$.
 We then have $v \oplus (v \wedge u') \le \bigvee\{u \oplus (u \wedge u') \mid u \in U\}$
 and applying $(\pi_1)_!$ we find $v \le \bigvee\{u \in U \mid u \between u'\} = \st(u', U) \in U^\star$.
 Therefore, $\widehat{\overline{U}} \subseteq U^\star$. It follows that $\U \subseteq \widehat{\overline{\U}}$
 and hence $\U = \widehat{\overline{\U}}$, as required.
 
 Finally, suppose $\E$ is an entourage uniformity on $X$. Then $\widelinehat{\E}$ has a base of entourages of the form $\widelinehat{E} = \bigvee\{u \oplus u \mid u \in \O X,\, u > 0,\, u \oplus u \le E\}$ for each entourage $E \in \E$.
 It is clear that $\widelinehat{E} \le E$ and hence $\E \subseteq \widelinehat{\E}$.
 
 For the other direction, let $F$ be an entourage such that $F \circ F \le E$ and let $G = F \wedge F^\mathrm{o} \in \E$.
 Suppose $u \oplus u' \le G$, where we may assume $u, u' > 0$ without loss of generality. We have $u \oplus u' \le G \le F$ and $u' \oplus u \le G^\mathrm{o} \le F$.
 Now in the internal logic, consider $x, z \in u$. Since $u' > 0$, we have $\exists y\colon X.\ y \in u'$. Thus, $(x,y) \in u \oplus u' \le F$ and $(y,z) \in u' \oplus u \le F$.
 It follows that $(x,z) \in F \circ F \le E$. So we have shown $u \oplus u \le E$. Similarly, $u' \oplus u' \in E$ and hence $(u \vee u') \oplus (u \vee u') = u \oplus u \vee u \oplus u' \vee u' \oplus u \vee u' \oplus u' \le E$. Thus, $u \oplus u' \le (u \vee u') \oplus (u \vee u') \le \widelinehat{E}$ and we can conclude that $G \le \widelinehat{E}$. It follows that $\widelinehat{\E} \le \E$ and hence $\widelinehat{\E} = \E$, as required.
\end{proof}

\begin{remark}
 We can also describe uniform embeddings in terms of covers. If $(X, \U)$ and $(Y,\V)$ are pre-uniform locales and $i\colon X \hookrightarrow Y$ is a locale embedding, then $i$ is a uniform embedding if and only if the covers ${\downarrow}\{ u \in i^*[V] \mid u > 0 \}$ for $V \in \V$ give a base for $\U$.
\end{remark}

Now in order to relate the entourage and covering approaches to \emph{uniform} locales, we must understand the relationship between their uniformly below relations.
\begin{lemma}\label{lem:uniformly_below_relation_covers_vs_entourages}
 With the definitions from the previous theorem, if $E$ is an entourage and $E \circ (a \oplus a) \le b \oplus b$ then $\st(a, \widehat{E}) \le b$.
 If $U$ is a uniform covering downset and $\st(a, U) \le b$, then $\overline{U} \circ (a \oplus a) \le b \oplus b$.
 Therefore, the uniformly below relation defined from an entourage uniformity and that defined by its corresponding covering uniformity coincide.
\end{lemma}
\begin{proof}
 Suppose $E \circ (a \oplus a) \le b \oplus b$. Recall that in the internal logic this is equivalent to $\exists y\colon X.\ y \in a \land (y,z) \in E \vdash_{z\colon X} z \in b$. We also have $\st(a, \widehat{E}) = \bigvee\{u \in \widehat{E} \mid a \between u\} = \bigvee\{u \mid u \oplus u\le E,\, a \between u\}$.
 Consider $u \in \O X$ such that $u \oplus u \le E$ and $a \between u$. Then if $z \in u$ in the internal logic, we can take $y \in a \wedge u$ and conclude that $z \in b$. So $u \le b$ and hence $\st(a, \widehat{E}) \le b$, as required.
 
 Now suppose $\st(a, U) \le b$. Then for any $u \in U$ such that $a \between u$, we have $u \le b$.
 We must show $y \in a \land (y,z) \in \overline{U} \vdash_{y,z\colon X} z \in b$ where $\overline{U} = \bigvee\{u \oplus u \mid u \in U\}$.
 By the join and existential quantification rules it is enough to show $\exists y\colon X.\ y \in a \land (y,z) \in u \oplus u \vdash_{z\colon X} z \in b$ for each $u \in U$.
 But this is just the assumption expressed in the internal logic and so we are done.
\end{proof}

Now combining \cref{thm:entourage_vs_cover_preuniformity} and \cref{lem:uniformly_below_relation_covers_vs_entourages} we obtain the following result.

\begin{theorem}
 The isomorphism of categories from \cref{thm:entourage_vs_cover_preuniformity} restricts to one between uniform locales via entourages and uniform locales via covers.
\end{theorem}

We can now provide the proof of \cref{prop:uniform_reflection}, which states that $\UnifLoc$ is a reflective subcategory of $\PreUnifLoc$, using the covering approach.

\begin{proof}[Proof of \cref{prop:uniform_reflection}]\label{proof:prop:uniform_reflection}
 Given a pre-uniform locale $X = (X, \U)$, we can consider the set $R \subseteq \O X$ of elements $b$ such that $b = \bigvee_{a \vartriangleleft b} a$. We claim that $R$ is a subframe of $\O X$.
 
 This is most easily understood by showing the map $r\colon b \mapsto \bigvee_{a \vartriangleleft b} a$ is a conucleus --- that is, a meet-preserving interior operator.
 Certainly, $r$ is monotone and deflationary. For idempotence, we consider $a \vartriangleleft b$. Then there is a $c$ such that $a \vartriangleleft c \vartriangleleft b$ and so $a \vartriangleleft c \le r(b)$. Hence, $r(b) \le r(r(b))$ and $r$ is idempotent.
 
 We have $r(1) = 1$, since $1 \vartriangleleft 1$. We must show $r(b) \wedge r(b') \le r(b \wedge b')$. It suffices to consider $a \vartriangleleft b$ and $a' \vartriangleleft b'$ and show that $a \wedge a' \vartriangleleft b \wedge b'$. But this is immediate from part (ii) of \cref{lem:uniformly_below_properties}.
 Thus, $r$ is indeed a conucleus and its set of fixed points $R$ is a subframe.
 
 We define $\Upsilon X$ to be the locale with underlying frame $R$ equipped with the final uniformity from the locale map given by $R \hookrightarrow \O X$.
 This uniformity can be described explicitly as $\{ U \subseteq R \mid {\downarrow} i[U] \in \U \}$ where $i$ is the subframe inclusion.
 It is easy to see this is indeed the final uniformity once we know it is a uniformity at all. The only nontrivial condition to check is that for each $U \subseteq R$ such that ${\downarrow} i[U] \in \U$ there is a strong cover $V$ such that ${\downarrow} i[V] \in \U$ and $V^\star \subseteq U$. Since $i$ is injective and hence strongly dense we know that $a > 0$ in $R$ if and only if $i(a) > 0$ in $\O X$. Thus, there is no difference between strength or the star operation with respect to the subframe versus the parent frame and so it is enough to show that covers of the from ${\downarrow} i[U]$ give a base for $\U$.
 
 Let $V$ be a uniform cover of $X$ and consider the downset $V' = \{u \mid u \vartriangleleft v \in V\}$. Take a uniform cover $W$ such that $W^\star \subseteq V$.
 Then for $w \in W$, we have $w \le \st(w, W) \in V$ and hence $w \in V'$. Thus, $W \subseteq V'$ and so $V'$ is a uniform cover. But note that $\bigvee_{u \vartriangleleft v} u$ lies in $R$ by idempotence of $r$ and set $U = {\downarrow}\{ \bigvee_{u \vartriangleleft v} u \mid v \in V \} \subseteq R$. It is now clear that $V' \subseteq {\downarrow} i[U] \subseteq V$ and so $U$ is uniform cover of $\Upsilon X$ and the covers of the form ${\downarrow} i[U]$ form a base for $\U$.
 
 The above arguments also quickly imply that the uniformly below relations on $R$ and $\O X$ agree. Thus, the uniformity on $\Upsilon X$ is admissible by the construction of $R$.
 
 We claim that the epimorphism of pre-uniform locales $\upsilon_X\colon X \to \Upsilon X$ induced by $i$ is the unit of an adjunction between $\Upsilon$ and the inclusion functor from $\UnifLoc$ into $\PreUnifLoc$.
 Consider the following diagram in $\PreUnifLoc$ where $Y$ is a \emph{uniform} locale. We must show $f^\flat$ exists. (It is unique since $\upsilon_X$ is epic.) 
 \begin{center}
   \begin{tikzpicture}[node distance=3.0cm, auto]
    \node (UX) {$\Upsilon X$};
    \node (X) [below of=UX] {$X$};
    \node (Y) [right of=X] {$Y$};
    \draw[->] (X) to node {$\upsilon_X$} (UX);
    \draw[->, dashed] (UX) to node {$f^\flat$} (Y);
    \draw[->] (X) to node [swap] {$f$} (Y);
   \end{tikzpicture}
 \end{center}
 Since $\upsilon_X$ is final, it is sufficient to check this factorisation on the level of the underlying locale maps. So we only need to show that the image of the frame homomorphism $f^*$ lies in $R \subseteq \O X$.
 
 As $Y$ is a uniform locale, every element of $\O Y$ satisfies $b = \bigvee_{a \vartriangleleft b} a$. Applying $f^*$ we obtain $f^*(b) = \bigvee_{a \vartriangleleft b} f^*(a) \le \bigvee_{a' \vartriangleleft f^*(b)} a' \in R$, where the inequality holds since $a \vartriangleleft b \implies f^*(a) \vartriangleleft f^*(b)$. The result follows.
\end{proof}

We end this section with some more basic results about uniform locales. The following result was already observed by Johnstone in \cite{johnstone1991constructive}.
\begin{proposition}
 Uniform locales are weakly regular.
\end{proposition}
\begin{proof}
 By \cref{lem:star_is_inflationary} we know that $a \vartriangleleft b$ implies that $a$ is weakly rather below $b$. Thus, weak regularity follows from the admissibility condition.
\end{proof}
\begin{remark}
 In fact this gives somewhat more. Classically under the assumption of the Axiom of Dependent Choice, uniform locales are not just regular, but completely regular. Without Dependent Choice (but still using classical logic) complete regularity is too strong a condition. However, there is a variant of completely regularity, called \emph{strong regularity} in \cite{banaschewski2003stone}, which still holds. A frame is strongly regular if every element $b$ can be expressed as $b = \bigvee_{a \vartriangleleft b} a$ for some interpolative relation $\vartriangleleft$ contained in $\prec$. This is implied by complete regularity, and is equivalent to it given Dependent Choice. Moreover, using classical logic, a locale is uniformisable if and only if it is strongly regular. The terminology becomes rather unfortunate in our setting: since the uniformly below relation interpolates, the above proof shows that every uniform locale is `strongly weakly regular'.
\end{remark}
\begin{corollary}\label{cor:strongly_dense_epic}
 Strongly dense uniform maps are epic in $\UnifLoc$.
\end{corollary}
\begin{proof}
 The proof proceeds exactly like the familiar one for Hausdorff spaces/locales.
\end{proof}

A full characterisation of uniformisable locales appears to be difficult and we do not attempt it here.
However, the following proposition does improve on Johnstone's result that completely regular (overt) locales are uniformisable.
\begin{proposition}
 Strongly regular overt locales are uniformisable.
\end{proposition}
\begin{proof}
 Let $X$ be an overt locale and let $\vartriangleleft$ be an interpolative relation contained in $\prec$ witnessing its strong regularity.
 We will take our subbasic uniform covering downsets to be ${\downarrow}\{ a^*, b \}$ for $a \vartriangleleft b$. Note that these are indeed covers since $a \vartriangleleft b$ implies $a \prec b$.
 
 We claim that the covers of the form ${\downarrow} \{u \in \bigwedge_{i \in I} C_i \mid u > 0\}$, where $I$ is finite and each $C_i$ is one of these subbasic covers, constitute a base for an admissible covering uniformity on $X$.
 To show this gives a uniformity it suffices to find for each subbasic cover ${\downarrow}\{a_0^*, a_1\}$ a finite set $S$ of subbasic covers such that $(\bigwedge S)^\star \subseteq {\downarrow} \{a_0^*, a_1\}$.
 We use the interpolativity of $\vartriangleleft$ to obtain $a_0 \vartriangleleft a_\frac{1}{3} \vartriangleleft a_\frac{2}{3} \vartriangleleft a_1$
 and consider $C = {\downarrow}\{a_0^*, a_\frac{1}{3}\} \cap {\downarrow}\{a_\frac{1}{3}^*, a_\frac{2}{3}\} \cap {\downarrow}\{a_\frac{2}{3}^*, a_1\}$.
 Writing each ${\downarrow}\{x, y\}$ as ${\downarrow}x \cup {\downarrow}y$ and using distributivity we find that $C = {\downarrow}\{a_\frac{1}{3}, a_0^* \wedge a_\frac{2}{3}, a_\frac{1}{3}^* \wedge a_1, a_\frac{2}{3}^*\}$.
 Observe that the first two elements $a_\frac{1}{3}$ and $a_0^* \wedge a_\frac{2}{3}$ are disjoint from the last element $a_\frac{2}{3}^*$, while the last two elements are disjoint from the first one.
 Thus, when $x$ is one of the first two elements we have $\st(x, C) \le a_\frac{1}{3} \vee (a_0^* \wedge a_\frac{2}{3}) \vee (a_\frac{1}{3}^* \wedge a_1) \le a_1$
 and when $x$ is one of the last two elements we have $\st(x, C) \le (a_0^* \wedge a_\frac{2}{3}) \vee (a_\frac{1}{3}^* \wedge a_1) \vee a_\frac{2}{3}^* \le a_0^*$.
 So we do have $C^\star \subseteq {\downarrow} \{a_0^*, a_1\}$ as required.
 
 Finally, we show admissibility. Suppose $a \vartriangleleft b$. Then $\{a^*, b\}$ is a uniform cover and we have $\st(a, \{a^*, b\}) \le b$ and hence $a$ is uniformly below $b$.
 Admissibility then follows from the assumption that $b = \bigvee_{a \vartriangleleft b} a$.
\end{proof}

\section{Completion}

We are now in a position to discuss completeness of uniform locales. The definition is similar to the classical case except we use weak closedness and strong density instead of ordinary closedness and density.

\begin{definition}
 A uniform locale $X$ is \emph{complete} if it is \emph{universally weakly closed} in the sense that whenever it occurs as a uniform sublocale of a uniform locale, it is a weakly closed sublocale.
 Equivalently, it is complete if every strongly dense uniform embedding $X \hookrightarrow Y$ is an isomorphism.
\end{definition}

Every uniform locale has a \emph{completion} --- that is, a (unique) complete uniform locale in which it is strongly densely uniformly embedded. As in the spatial setting, the completion may be constructed by means of \emph{Cauchy filters}.

\begin{definition}
 A \emph{Cauchy filter} on a uniform locale $(X, \U)$ is a filter $F$ on $\O X$ such that
 \begin{itemize}
  \item $F$ only contains positive elements,
  \item $F$ contains an open from every uniform cover.
 \end{itemize}
 We say a Cauchy filter $F$ is \emph{regular} if for every $a \in F$ there is a $b \in F$ such that $b \vartriangleleft a$.
\end{definition}

The regular Cauchy filters on $X$ will be the points of the completion of $X$. Thus, we consider the classifying locale $\Cvar X$ of the theory of regular Cauchy filters on $X$.
This classifying locale is given by a presentation with a generator $[a \in F]$ for each $a \in \O X$ and the following relations:
\begin{enumerate}
 \item $[1 \in F] = 1$,
 \item $[a \wedge b \in F] = [a \in F] \wedge [b \in F]$,
 \item $[a \in F] \le \bigvee\{1 \mid a > 0\}$,
 \item $\bigvee_{u \in U} [u \in F] = 1$ for each $U \in \U$,
 \item $[a \in F] \le \bigvee_{b \vartriangleleft a} [b \in F]$.
\end{enumerate}
By the universal property of the classifying locale, the identity morphism on $\Cvar X$ corresponds to a `$\Cvar X$-indexed regular Cauchy filter on $X$' sending $a \in \O X$ to $[a \in F] \in \O\Cvar X$.
There is also a locale embedding $\gamma\colon X \hookrightarrow \Cvar X$ given by the frame homomorphism $[a \in F] \mapsto a$ (corresponding to the identity $X$-indexed regular Cauchy filter on $X$).
In fact, it is not hard to see that the former map is the right adjoint of the latter. (For the nontrivial direction, write a general element of $\O\Cvar X$ as $\bigvee_\alpha [a_\alpha \in F]$ and break the inequality up into a part for each $\alpha$.)
Condition (iii) then implies that $\gamma_*{!}^*(p) \le {!}^*(p)$ for all $p \in \Omega$ and hence $\gamma$ is strongly dense.

It will also be useful to define the \emph{classifying locale $\Cauchy X$ of all Cauchy filters on $X$}. The presentation is similar to that of $\Cvar X$, but without the regularity condition (v).
As above, we have a canonical strongly dense embedding  $\digamma\colon X \hookrightarrow \Cauchy X$ defined by $\digamma^*\colon [a \in F] \mapsto a$
and whose right adjoint is a $\Cauchy X$-indexed Cauchy filter on $X$ sending $a$ to $[a \in F]$. Of course, we have $\digamma^* = \gamma^* \rho^*$ and $\rho^*\digamma_* = \gamma_*$ where $\rho$ is the natural embedding $\Cvar X \hookrightarrow \Cauchy X$.

We can use $\digamma_*$ to define a pre-uniform structure on $\Cauchy X$.
\begin{lemma}\label{lem:preuniform_structure_on_CL}
 Let $(X, \U)$ be a uniform locale. The downsets of the form ${\downarrow}\digamma_*[U]$ for $U \in \U$ form a base for a covering uniformity on $\Cauchy X$.
 Moreover, $\digamma\colon X \hookrightarrow \Cauchy X$ is a strongly dense uniform embedding.
\end{lemma}
\begin{proof}
 First note that $\Cauchy X$ is overt, since $X$ is overt and $\digamma\colon X \to \Cauchy X$ is strongly dense.
 
 The above downsets indeed are covers by condition (iv). It is also easy to see that they form a filter base.
 
 For the star-refinement axiom, consider ${\downarrow}\digamma_*[U]$. There is a strong $V \in \U$ such that $V^\star \le U$. We claim $({\downarrow}\digamma_*[V])^\star \subseteq {\downarrow}\digamma_*[U]$.
 It suffices to show $\st(\digamma_*(v), \digamma_*[V]) \le \digamma_*(\st(v, V)) \in {\downarrow}\digamma_*[U]$ for all $v \in V$. Consider $w \in V$ such that $\digamma_*(v) \between \digamma_*(w)$. Since $\digamma$ is strongly dense, we then have $v \wedge w = \digamma^*\digamma_*(v \wedge w) = \digamma^*(\digamma_*(v) \wedge \digamma_*(w)) > 0$ and hence $v \between w$. Thus, $w \le \st(v, V)$ so that $\digamma_*(w) \le \digamma_*(\st(v, V))$ and the claim follows.
 
 Finally, since $\digamma_*$ is a section of $\digamma^*$, it preserves positive elements. To see this explicitly, suppose $a \in \O X$ is positive and $\digamma_*(a) \le \bigvee A$. Applying $\digamma^*$ we have $a = \digamma^*\digamma_*(a) \le \bigvee \digamma^*[A]$. Hence $\digamma^*[A]$ is inhabited and so is $A$. Thus, $\digamma^*(a)$ is positive.
 It follows that $\digamma_*[V]$ as defined above is strong since $V$ is. Thus, we have shown that we do have a base for a uniformity on $\Cauchy X$.
 The map $\digamma$ is then a uniform embedding by construction, since $\digamma^*\digamma_* = \id_{\O X}$.
\end{proof}

We obtain a uniform structure on $\Cvar X$ in a similar way.

\begin{corollary}\label{lem:uniform_structure_on_CL}
 Let $(X, \U)$ be a uniform locale. The downsets of the form ${\downarrow}\gamma_*[U]$ for $U \in \U$ form a base for an admissible covering uniformity on $\Cvar X$.
 Moreover, $\gamma\colon X \hookrightarrow \Cvar X$ is a strongly dense uniform embedding.
\end{corollary}
\begin{proof}
 As above we have that $\Cvar X$ is overt, since $X$ is overt and $\gamma\colon X \to \Cvar X$ is strongly dense.
 Now since $\rho^*\digamma_* = \gamma_*$, the downsets ${\downarrow}\gamma_*[U]$ form a base for the uniformity on $\Cvar X$ inherited as a sublocale of $\Cauchy X$.
 Furthermore, $\gamma$ is a uniform embedding as before.
 
 It only remains to show admissibility. As we showed above for $\digamma$ we have $\st(\gamma_*(v), \gamma_*[V]) \le \gamma_*(\st(v, V))$.
 Applying $\gamma_*$ we then see that $\st(v, V) \le u$ implies $\st(\gamma_*(v), \gamma_*[V]) \le \gamma_*(u)$ and so $v \vartriangleleft u$ in $\O X$ implies $\gamma_*(v) \vartriangleleft \gamma_*(u)$ in $\O \Cvar X$.
 It is enough to show admissibility for the basic opens $\gamma_*(u) = [u \in F]$ and by condition (v) we have $\gamma_*(u) \le \bigvee_{v \vartriangleleft u} \gamma_*(v) \le \bigvee_{\gamma_*(v) \vartriangleleft \gamma_*(u)} \gamma_*(v)$, as required.
\end{proof}

The uniformity on $\Cauchy X$ is not in general admissible, while the uniformity on $\Cvar X$ is. The following proposition explains the relationship between these pre-uniform locales.
\begin{proposition}\label{prop:uniform_reflection_of_Cauchy_X}
The uniform locale $\Cvar X$ is the uniform reflection of $\Cauchy X$. Moreover, the frame homomorphism corresponding to the unit $\upsilon_{\Cauchy X}\colon \Cauchy X \to \Cvar X$ is left adjoint to $\rho^*$.
\end{proposition}
\begin{proof}
 We claim there is a well-defined frame map $r\colon \O\Cvar X \to \O\Cauchy X$ sending $[a \in F]$ to $\bigvee_{b \vartriangleleft a} [b \in F]$.
 By part (ii) of \cref{lem:uniformly_below_properties} we see it satisfies relations (i) and (ii) and it satisfies (iii) simply due to the similar relation on $\O\Cauchy X$.
 For condition (iv) we can use that star-refinement axiom to show that $\{b \mid b \vartriangleleft a \in U\}$ is a uniform cover whenever $U$ is.
 Finally, relation (v) holds since $\vartriangleleft$ interpolates.
 
 We now show that $r$ is left adjoint to $\rho^*$. It is clear that $\rho^* r = \id_{\O\Cvar X}$ by condition (v) on $\O\Cvar X$.
 On the other hand, $r\rho^*([a \in F]) = \bigvee_{b \vartriangleleft a} [b \in F] \le [a \in F]$ and so $r\rho^* \le \id_{\O\Cauchy X}$, as required.
 
 By the universal property of the uniform reflection, the image of $r\colon \O\Cvar X \to \O\Cauchy X$ lies in the the subframe $\O\Upsilon\Cauchy X$ of elements $B$ such that $B \le \bigvee_{A \vartriangleleft B} A$.
 We must show that every such element lies in the image of $r$.
 
 Suppose $B \le \bigvee_{A \vartriangleleft B} A$. Then since the elements $[a \in F]$ form a base, we have $B \le \bigvee_{[a \in F] \vartriangleleft B} [a \in F]$.
 Now suppose $[a \in F] \vartriangleleft B$. Then $\st([a \in F], \digamma_*[U]) \le B$ for some uniform cover $U$ on $X$. Explicitly, $\st([a \in F], \digamma_*[U]) = \bigvee_{u \in U,\, a \between u} [u \in F]$.
 Now let $V$ be a uniform cover such that $V^\star \le U$ and take $v \in V$ such that $v \between a$.
 Then $\st(a \wedge v, V) \le \st(v, V) \in U$, and clearly $\st(a \wedge v, V) \between a$, so that $[\st(a \wedge v, V) \in F] \le \st([a \in F], \digamma_*[U]) \le B$.
 Thus, $[a \wedge v \in F] \vartriangleleft [b \in F] \le B$ for some $b \in \O X$
 and so we have $[a \wedge v \in F] \le \bigvee_{[c \in F] \vartriangleleft [b \in F] \le B} [c \in F] \le \bigvee_{[b \in F] \le B} \bigvee_{c \vartriangleleft b} [c \in F] = r\rho^*(B)$,
 where the second inequality holds since $\digamma^*$ preserves the uniformly below relation.
 
 Taking the join over all such $v$ we then find $\bigvee_{v \in V,\, v \between a} [a \wedge v \in F] \le r\rho^*(B)$.
 Now note that $\bigvee_{v \in V,\, v \between a} [a \wedge v \in F] = \bigvee_{v \in V} [a \wedge v \in F] = [a \in F] \wedge \bigvee_{v \in V} [v \in F] = [a \in F]$ where the first equality follows from condition (iii) for $\Cauchy X$, the second equality is from condition (ii) and the third equality is from condition (iv).
 Therefore, $[a \in F] \le r\rho^*(B)$ and so $r\rho^*(B) \le B \le \bigvee_{[a \in F] \vartriangleleft B} [a \in F] \le r\rho^*(B)$.
 Thus, $B$  is in the image of $r$ as required.
\end{proof}

\begin{remark}
 The adjunction of the frame homomorphisms above gives $\rho \dashv \upsilon_{\Cauchy X}$ in $\Loc$. This can be understood as a pointfree manifestation of the fact that regular Cauchy filters are the same as \emph{minimal} Cauchy filters, with the composite $\rho\upsilon_{\Cauchy X}$ associating each Cauchy filter to the smallest Cauchy filter contained in it.
\end{remark}

We now observe that $\Cauchy$ and $\Cvar$ are functorial.
\begin{proposition}\label{prop:Cauchy_X_is_functorial}
 The construction of the locale of Cauchy filters gives rise to a functor $\Cauchy\colon \UnifLoc \to \PreUnifLoc$. Furthermore, the maps $\digamma\colon X \to \Cauchy X$ assemble into a natural transformation from the inclusion $\UnifLoc \hookrightarrow \PreUnifLoc$ to $\Cauchy$.
\end{proposition}
\begin{proof}
 Let $f\colon X \to Y$ be a morphism of uniform locales. Recall that $\digamma^X_*\colon \O X \to \O\Cauchy X$ is an $\Cauchy X$-indexed Cauchy filter on $X$,
 from which it easily follows that $\digamma^X_* f^*$ is an $\O\Cauchy X$-indexed Cauchy filter on $Y$.
 By the universal property of $\Cauchy Y$ we obtain a frame homomorphism from $\O\Cauchy Y$ to $\O\Cauchy X$. Explicitly this sends $[a \in F]$ to $[f^*(a) \in F]$.
 We define $\Cauchy f$ to be a corresponding locale map.
 
 It is now easy to see this definition respects identities and composition
 and that $\digamma$ is natural.
\end{proof}

\begin{corollary}\label{cor:CX_is_functorial}
 Similarly, the regular Cauchy filter locale yields a functor $\Cvar\colon \UnifLoc \to \UnifLoc$ and the maps $\gamma\colon X \to \Cvar X$ give a natural transformation from the identity to $\Cvar$.
\end{corollary}
\begin{proof}
 Simply set $\Cvar = \Upsilon \circ \Cauchy$ and note that $\gamma$ is equal to $\Upsilon\digamma$ (up to isomorphism).
\end{proof}

We shall now attempt to show $\Cvar X$ is the completion of $X$.
To see this we note that some of the results that we have seen to hold for $\gamma$ actually hold for any strongly dense uniform embedding.
We will proceed in a similar manner to \cite[Chapter VIII]{picado2012book} for the next few results.
Let us begin by proving an analogue of \cref{lem:uniform_structure_on_CL}.

\begin{lemma}\label{lem:base_from_strongly_dense_uniform_embedding}
 If $j\colon (X, \U) \hookrightarrow (Y, \V)$ is a strongly dense embedding of pre-uniform locales, then the downsets of the form ${\downarrow}j_*[U]$ for $U \in \U$ form a base for $\V$.
\end{lemma}
\begin{proof}
 We first observe that the downsets ${\downarrow}j_*[U]$ are $\V$-uniform covering downsets. Since $j$ is a uniform embedding, every $U \in \U$ contains ${\downarrow} j^*[V]$ for some $V \in \V$.
 We then have ${\downarrow}j_*[U] \supseteq {\downarrow} j_*j^*[V] \supseteq V$ and hence ${\downarrow}j_*[U] \in \V$.
 
 Now take $V \in \V$. There is a $W \in \V$ such that $W^\star \subseteq V$. Then ${\downarrow} j^*[W] \in \U$ and we claim ${\downarrow} j_*j^*[W] \subseteq V$.
 An element of $j_*j^*[W]$ is of the form $j_*j^*(w)$ for $w \in W$.
 Note that if $j_*j^*(w) \between a$, then by strong density of $j$ we have $j^*(w) = j^*j_*j^*(w) \between j^*(a)$, which in turn means $w \between a$.
 Thus, the open sublocale $j_*j^*(w)$ is contained in $\wkclo{w}$.
 But $w \in W$ satisfies $\st(w, W) \le v$ for some $v \in V$. Now by \cref{lem:star_is_inflationary} we have $j_*j^*(w) \le \wkclo{w} \le v$ and hence $j_*j^*(w) \in V$, as claimed.
\end{proof}

Now recall that the right adjoint $\gamma_*$ is an indexed regular Cauchy filter. In fact, this is true of every strongly dense uniform embedding, which helps explain why there is a link between Cauchy filters and completeness in the first place.
\begin{lemma}\label{lem:right_adjoint_strongly_dense_embedding_cauchy}
 Let $j\colon (X, \U) \hookrightarrow (Y, \V)$ be a strongly dense embedding of uniform locales. Then the right adjoint $j_*\colon \O X \to \O Y$ is an $Y$-indexed regular Cauchy filter.
\end{lemma}
\begin{proof}
 Since $j_*$ preserves meets it satisfies conditions (i) and (ii) and strong density means $j_*(a) \le {!}^*{\exists}(a)$ and hence condition (iii) holds.
 
 To see condition (iv) holds recall that $j$ is a uniform embedding and so each uniform covering downset $U \in \U$ satisfies $U \supseteq {\downarrow} j^*[V]$ for some $V \in \V$.
 We then have $\bigvee_{u \in U} j_*(u) \ge \bigvee_{v \in V} j_*(j^*(v)) \ge \bigvee_{v \in V} v = 1$, as required.
 
 Now note that if $b \vartriangleleft j_*(a)$, then $j^*(b) \vartriangleleft j^*j_*(a) = a$ and hence
 $j_*(a) = \bigvee_{b \vartriangleleft j_*(a)} b \le \bigvee_{j^*(b) \vartriangleleft a} b \le \bigvee_{j^*(b) \vartriangleleft a} j_*j^*(b) \le \bigvee_{b' \vartriangleleft a} j_*(b')$,
 which is the regularity condition (v).
\end{proof}

We are now able to prove that $\Cvar X$ is, in a sense, the largest uniform locale into which $X$ can be strongly densely embedded.

\begin{proposition}\label{prop:dense_uniform_embedding_factors_through_gamma}
 Let $j\colon X \hookrightarrow Y$ be any strongly dense embedding of uniform locales. Then $\gamma\colon X \hookrightarrow \Cvar X$ factors uniquely through $j$ to give a strongly dense uniform embedding $k\colon Y \hookrightarrow \Cvar X$.
\end{proposition}
\begin{proof}
 By \cref{lem:right_adjoint_strongly_dense_embedding_cauchy} we have that $j_*\colon \O X \to \O Y$ is an indexed regular Cauchy filter. Then by the universal property of $\Cvar X$ this gives a locale map $k\colon Y \to \Cvar X$
 whose corresponding frame homomorphism is defined by $[a \in F] \mapsto j_*(a)$. This indeed composes with $j^*$ to give $\gamma^*$, since $j^*j_* = \id_{\O X}$.
 
 Pulling back a basic uniform cover on $\Cvar X$ along $k$ gives ${\downarrow}\{k^*\gamma_*(u) \mid u \in U\}$ for $U$ a uniform cover on $X$. But note that $k^*\gamma_* = j_*$ and so applying \cref{lem:base_from_strongly_dense_uniform_embedding} to $j$ these give a base of uniform covers of $Y$. Hence $k$ is uniform and initial.
 Now to show that $k$ is a uniform embedding it remains to prove $k^*$ is surjective. Observe that if $b \vartriangleleft a$ in $\O Y$ then $b \le \bigvee\{v \in V \mid v \between b\} \le a$ for some basic uniform cover $V$ of $Y$ and hence each $a \in \O Y$ is a join of elements of the form $k^*\gamma_*(u)$.
 Consequently, $k^*$ is indeed surjective.
 Finally, $k$ is strongly dense since $\gamma$ is and the map $k$ is unique since $j$ is epic by \cref{cor:strongly_dense_epic}.
\end{proof}

It follows that $\Cvar X$ is indeed the completion of $X$.

\begin{theorem}\label{thm:completion}
 The uniform locale $\Cvar X$ is the unique completion of $X$.
\end{theorem}
\begin{proof}
 We first show that $\Cvar X$ is complete. Suppose $e\colon \Cvar X \hookrightarrow Y$ is a strongly dense uniform embedding. Then by \cref{prop:dense_uniform_embedding_factors_through_gamma} we have that $e\gamma$ factors through $\gamma$ to give a uniform map $f\colon Y \to \Cvar X$ such that $\gamma = fe\gamma$. But $\gamma$ is strongly dense and hence epic, so that $fe = \id_{\Cvar X}$ and $e$ is a split monomorphism. But $e$ is also an epimorphism and thus an isomorphism. Therefore, $\Cvar X$ is complete.
 
 Now we show uniqueness. Suppose $e'\colon X \hookrightarrow C$ is a completion of $X$. Then by \cref{prop:dense_uniform_embedding_factors_through_gamma} there is a strongly dense uniform embedding $f'\colon C \hookrightarrow \Cvar X$ such that $f' e' = \gamma$. This is an isomorphism by the completeness of $C$.
\end{proof}

\begin{theorem}
 Complete uniform locales form a reflective subcategory $\CUnifLoc$ of $\UnifLoc$ with $\Cvar$ as the reflector and $\gamma$ as the unit.
\end{theorem}
\begin{proof}
 By \cref{thm:completion}, $\Cvar X$ is complete. Moreover, $\gamma_Y$ is an isomorphism for a complete uniform locale $Y$, since $\gamma_Y$ is a strongly dense uniform embedding.
 Thus, any uniform map $f\colon X \to Y$ factors through $\gamma_X$ to give $\gamma_Y^{-1}\Cvar f$. Finally, this factorisation is unique since $\gamma_Y$ is epic.
\end{proof}

 We will occasionally wish to talk about the completion of a pre-uniform locale. This may be defined simply as the completion of its uniform reflection.
 On the other hand, the definition of $\Cauchy X$ works equally well when $X$ is only a pre-uniform locale and its fundamental properties still go through. With no essential modification the proof of \cref{prop:uniform_reflection_of_Cauchy_X} then shows that $\Upsilon \Cauchy X \cong \Cvar \Upsilon X$ in this general situation.
 Moreover, we have $\gamma_{\Upsilon X} \upsilon_X = \upsilon_{\Cauchy X} \digamma_X$.

Even with this generalised notion of completion we have the following result.
\begin{proposition}\label{prop:completion_preserves_products}
 The completion of pre-uniform locales preserves finite products.
\end{proposition}
\begin{proof}
 The terminal pre-uniform locale is already a complete uniform locale, since reflective subcategories are closed under limits in the parent category.
 
 To show that $\Cvar\Upsilon$ preserves binary products, it suffices to show that the canonical map from $\Upsilon(X \times Y)$ to $\Cvar\Upsilon X \times \Cvar\Upsilon Y$ is a strongly dense uniform embedding.
 The latter is complete, since as a reflective subcategory, complete uniform locales are closed under products.
 
 Certainly the unit $\upsilon_X\colon X \to \Upsilon X$ is strongly dense since it is an epimorphism of locales. Also, we already know that $\gamma_{\Upsilon X}\colon \Upsilon X \to \Cvar \Upsilon X$ is strongly dense.
 Therefore, the composite $\gamma_{\Upsilon X}\upsilon_X\colon X \to \Cvar \Upsilon X$ is strongly dense (and similarly for $Y$). It then follows that the product map $X \times Y \to \Cvar \Upsilon X \times \Cvar \Upsilon Y$ is strongly dense.
 This factors through the unit map $\upsilon_{X \times Y}\colon X \times Y \to \Upsilon(X \times Y)$ to give the map in question. Now as required, this map is strongly dense since it is the second factor of a strongly dense map.
\end{proof}

The functoriality of $\Cvar$ means that uniform maps between pre-uniform locales lift to (uniform) maps between their completions.
However, these are not the only maps that lift to (not-necessarily-uniform) maps between the completions.
\begin{lemma}\label{prop:lift_to_completion}
 A general locale map $f\colon X \to Y$ between pre-uniform locales lifts to give a (unique) locale map $\widetilde{f}\colon \Cvar\Upsilon X \to \Cvar\Upsilon Y$ such that $\widetilde{f} \gamma_{\Upsilon X}\upsilon_X = \gamma_{\Upsilon Y} \upsilon_Y f$
 if and only if $(\gamma_{\Upsilon X})_* (\upsilon_X)_* f^*$ sends uniform covers to covers.
\end{lemma}
\begin{proof}
 The reverse direction of the proof proceeds as in the proof of \cref{prop:Cauchy_X_is_functorial}.
 The composite map $(\gamma_{\Upsilon X})_* (\upsilon_X)_* f^*$ is an $\Cvar X$-indexed Cauchy filter on $Y$
 and thus lifts to give a frame homomorphism from $\O \Cauchy Y$ to $\O \Cvar X$. Then the universal property of the uniform reflection of $\Cauchy Y$ (together with the isomorphism $\Upsilon \Cauchy X \cong \Cvar \Upsilon X$) gives the desired lift $\widetilde{f}\colon \Cvar \Upsilon X \to \Cvar \Upsilon Y$.
 
 For the forward direction consider the following calculation.
 \begin{displaymath}
  (\gamma_{\Upsilon X})_* (\upsilon_X)_* f^* \upsilon_Y^* \gamma_{\Upsilon Y}^* = (\gamma_{\Upsilon X})_* (\upsilon_X)_* \upsilon_X^* \gamma_{\Upsilon X}^* \widetilde{f}^* \ge \widetilde{f}^*
 \end{displaymath}
 Certainly, the frame homomorphism $\widetilde{f}^*$ sends all covers to covers and hence so does $(\gamma_{\Upsilon X})_* (\upsilon_X)_* f^* \upsilon_Y^* \gamma_{\Upsilon Y}^*$.
 But we also know that $(\gamma_{\Upsilon Y})_*$ and $(\upsilon_Y)_*$ preserve uniform covers and hence $(\gamma_{\Upsilon X})_* (\upsilon_X)_* f^* \upsilon_Y^* \gamma_{\Upsilon Y}^* (\gamma_{\Upsilon Y})_* (\upsilon_Y)_* = (\gamma_{\Upsilon X})_* (\upsilon_X)_* f^* \upsilon_Y^* (\upsilon_Y)_*$ sends uniform covers to covers.
 Now in the proof of \namecref{prop:uniform_reflection}~\hyperref[proof:prop:uniform_reflection]{\ref*{prop:uniform_reflection}}
 we showed uniform covers of the form ${\downarrow} \upsilon_Y^* (\upsilon_Y)_*[U]$
 form a base for the uniformity on $Y$ and hence $(\gamma_{\Upsilon X})_* (\upsilon_X)_* f^*$ also sends uniform covers to covers, as required.
\end{proof}
This can also be expressed in terms of entourages.
\begin{corollary}\label{prop:lift_to_completion_via_entourages}
 A locale map $f\colon X \to Y$ between pre-uniform locales lifts to a locale map $\widetilde{f}\colon \Cvar\Upsilon X \to \Cvar\Upsilon Y$ as above if and only if $(\gamma_{\Upsilon X}\upsilon_X \times \gamma_{\Upsilon X}\upsilon_X)_* (f \times f)^*$ sends entourages on $Y$ to reflexive relations on $\Cvar\Upsilon X$.
\end{corollary}
\begin{proof}
 We must simply show that $(\gamma\upsilon \times \gamma\upsilon)_* (f \times f)^*$ sends entourages to reflexive relations if and only if $\gamma_*\upsilon_* f^*$ sends uniform covers to covers.
 
 The proof of the forward direction proceeds almost exactly as in the proof that the map $(X, \E) \mapsto (X, \widehat{\E})$ is functorial in \cref{thm:entourage_vs_cover_preuniformity} (and using $u \le \gamma_*\upsilon_*f^*(v) \iff \upsilon^*\gamma^*(u) \le f^*(v)$ as appropriate).
 
 Conversely, suppose $\gamma_*\upsilon_*$ sends uniform covers to covers.
 Using the same notation as in \cref{thm:entourage_vs_cover_preuniformity} (and applying \cref{lem:adjoint_of_product_map}) we have
 \begin{align*}
  \Delta^* (\gamma\upsilon \times \gamma\upsilon)_* (f \times f)^*(\overline{V}) &= \bigvee\{\gamma_*\upsilon_*(a) \wedge \gamma_*\upsilon_*(b) \mid a \oplus b \le (f \times f)^*(\overline{V})\} \\
          &= \bigvee\{\gamma_*\upsilon_*(c) \mid c \oplus c \le (f \times f)^*(\overline{V})\} \\
          &= \bigvee\{\gamma_*\upsilon_*(c) \mid c \oplus c \le \bigvee\{f^*(v) \oplus f^*(v) \mid v \in V\} \} \\
          &\ge \bigvee\{\gamma_*\upsilon_*(c) \mid c \le f^*(v),\, v \in V \} \\
          &= \bigvee\{\gamma_*\upsilon_*f^*(v) \mid v \in V \} \\
          &= \bigvee ({\downarrow}\gamma_*\upsilon_*f^*[V]) \\
          &= 1.
 \end{align*}
 Thus, $(\gamma\upsilon \times \gamma\upsilon)_* (f \times f)^*$ indeed sends entourages to reflexive relations, as required.
\end{proof}

It is sometimes convenient to be able to describe the completion of a (pre-)uniform locale $X$ in terms of basic opens and basic uniformities on $X$ instead of all elements of $\O X$.
This was done in the metric case by Henry in \cite{Henry2016}. We conclude this section with a discussion of how to do this in our setting.

Let $X$ be a locale and let $\mathcal{B}$ be a collection of covers of $X$ by positive opens such that $\{{\downarrow} C \mid C \in \mathcal{B}\}$ is a base for a covering uniformity $\U$ on $X$. Set $B = \bigcup \mathcal{B}$.

Note that every element of $\O\Upsilon X$ is a join of elements of $B$. For $a \in \O\Upsilon X$ we have $a \le \bigvee_{b \vartriangleleft a} b$. Now $b \vartriangleleft a$ means $b \le \st(b, C) = \bigvee \{c \in C \mid c \between b\} \le a$ for some $C \in \mathcal{B}$ and so $a \le \bigvee\{c \mid C \in \mathcal{B},\, c \in C,\, b \in \O X,\, \st(b, C) \le a,\, c \between b \} \le a$.

By replacing each element $c$ of each $C \in \mathcal{B}$ by $\bigvee_{c' \vartriangleleft c} c'$ we may assume without loss of generality that $B \subseteq \O\Upsilon X$ and hence that $B$ is a base for $\O\Upsilon X$.

\begin{proposition}\label{prop:completion_from_a_base}
 Given $X$, $\mathcal{B}$ and $B$ as defined above, the completion $\Cvar \Upsilon X$ has an alternative presentation with a generator $[b \in F_B]$ for each $b \in B$ and the following relations:
 \begin{enumerate}[a)]
  \item $[a \in F_B] \le [b \in F_B]$ for $a \le b$,
  \item $[a \in F_B] \wedge [b \in F_B] \le \bigvee_{c \in B \cap {\downarrow} a \cap {\downarrow} b} [c \in F_B]$,
  \item $\bigvee_{c \in C} [c \in F_B] = 1$ for each $C \in \mathcal{B}$,
  \item $[a \in F_B] \le \bigvee_{b \in B,\, b \vartriangleleft a} [b \in F_B]$.
 \end{enumerate}
\end{proposition}
\begin{proof}
 We first note that every generator $[a \in F]$ in the original presentation for $\Cvar X$ satisfies $[a \in F] = \bigvee_{b \in B \cap {\downarrow} a} [b \in F]$.
 To see this, recall that $[a \in F] \le \bigvee_{a' \vartriangleleft a} [a' \in F]$ by regularity and that $a' \vartriangleleft a$ means $\st(a', C) \le a$ for some $C \in \mathcal{B}$.
 Then by Cauchiness we have $[a' \in F] \le [a' \in F] \wedge \bigvee_{c \in C} [c \in F] = \bigvee_{c \in C} [a' \wedge c \in F]$. Putting these together, we find \[[a \in F] \le \bigvee \{ [a' \wedge c \in F] \mid C \in \mathcal{B},\, c \in C,\, a' \in \O X,\, \st(a', C) \le a\}.\]
 Now by properness (that is, condition (iii) for the presentation) we can restrict the join to the elements $[a' \wedge c \in F]$ such that $a' \between c$.
 But $a' \between c$ and $c \in C$ give $c \le \st(a', C) \le a$ and hence
 $[a \in F] \le \bigvee \{ [c \in F] \mid C \in \mathcal{B},\, c \in C,\, c \le a\} \le \bigvee_{c \in B \cap {\downarrow} a} [c \in F]$, as claimed.
 
 We can use this to show the generators $[b \in F]$ for $b \in B$ indeed satisfy the above relations required for $[b \in F_B]$.
 That (a) and (b) hold is immediate from (ii) and the claim above, while relation (c) follows from (iv).
 Finally, relation (d) follows from the regularity condition (v) and the above claim.
 
 It remains to define $[a \in F]$ in terms of $[b \in F_B]$ show that relations (i)--(v) for the original presentation follow from (a)--(d).
 The claim suggests defining $[a \in F]$ to be $\bigvee_{b \in B \cap {\downarrow} a} [b \in F]$.
 Condition (i) says that $\bigvee_{b \in B} [b \in F_B] = 1$, which holds by (c) and the fact that, as a base, $\mathcal{B}$ is inhabited.
 The ``$\le$'' direction of (ii) follows from (a), while for the ``$\ge$'' direction follows from (b). Condition (iii) comes from the fact that every $b \in B$ is positive.
 Finally, (iv) and (v) follow from (c) and (d) respectively.
\end{proof}

\section{Examples and applications}

In this section we discuss the relationship between metric and uniform locales and some applications of uniform locales to localic algebra.

\subsection{Metric locales}

Classically, perhaps the most important example of uniform spaces is given by metric spaces.
As one might expect, in the constructive setting metric locales also have uniform structures. Our definition of metric locales follows that of Henry in \cite{Henry2016}.

Metric locales take values in the nonnegative extended upper reals $\overleftarrow{\R_{\ge 0}^\infty}$ --- the classifying locale of extended upper Dedekind cuts on the set of positive rationals $\Q_+$. Explicitly, a presentation of $\O \overleftarrow{\R_{\ge 0}^\infty}$ has a generator $[0, q)$ for each $q \in \Q_+$ subject to the relations $[0, q) = \bigvee_{p < q} [0, p)$.
\begin{definition}
 A \emph{pre-metric locale} is an overt locale $X$ equipped with a locale map $d\colon X \times X \to \overleftarrow{\R_{\ge 0}^\infty}$
 such that in the internal logic we have
 \begin{itemize}
  \item \makebox[140pt][l]{$d(x,x) = 0$}                 $(x \colon X)$,
  \item \makebox[140pt][l]{$d(x,y) = d(y,z)$}            $(x,y \colon X)$,
  \item \makebox[140pt][l]{$d(x,z) \le d(x,y) + d(y,z)$} $(x,y,z \colon X)$,
 \end{itemize}
 where $\le$ is interpreted as the \emph{reverse} of the order enrichment.
\end{definition}

If $(X, d)$ is a metric locale, we now attempt to define a uniformity $\E_d$ on $X$ with basic entourages of the form $E_q = d^*([0, q))$ for each $q \in \Q_+$.
\begin{lemma}
 This indeed gives a pre-uniform locale $(X, \E_d)$.
\end{lemma}
\begin{proof}
 Since $E_p \le E_q$ for $p \le q$, it is clear that the basic entourages $E_q$ form a filter base.
 For the reflexivity condition, we know that in the internal logic we have $d(x,x) = 0 < q$ and hence $(x,x) \in E_q$, as required.
 The symmetry condition follows from the fact that $E_q$ is symmetric by the symmetry condition on $d$.
 Finally, we will show $E_{q/2} \circ E_{q/2} \le E_q$.
 In the internal logic, $(x,z) \in E_{q/2} \circ E_{q/2}$ means $\exists y\colon X.\ (x,y) \in E_{q/2} \land (y,z) \in E_{q/2}$.
 That is, $d(x,y) < q/2$ and $d(y,z) < q/2$. By the triangle inequality it follows that $d(x,z) \le d(x,y) + d(x,y) < q/2 + q/2 = q$.
 Thus, $(x,z) \in E_q$ and we have shown the desired inclusion.
\end{proof}

This construction also interacts well with morphisms of pre-metric spaces.

\begin{definition}
 A map $f\colon X \to Y$ between pre-metric locales is \emph{nonexpansive} if we have $\vdash_{x,y\colon X} d^Y(f(x),f(y)) \le d^X(x,y)$ in the internal logic --- that is, if $E_q^X \le (f \times f)^*(E_q^Y)$ for all $q \in \Q_+$.
\end{definition}

From this definition it is clear that a nonexpansive map is uniform with respect to the induced pre-uniformities. Thus, the association of the above metric uniformity to a metric locale to a uniform locale gives a forgetful functor from the category $\PreMetLoc$ of pre-metric locales and nonexpansive maps to $\PreUnifLoc$, which commutes with the obvious forgetful functors into $\OLoc$.

There is a relation $\vartriangleleft_q$ on pre-metric locales, the definition of which is precisely equivalent to $a \vartriangleleft_q b \iff (a \oplus a) \circ E_q \le b \oplus b$.
Then the definition of $\vartriangleleft$ by $a \vartriangleleft b \iff \exists q \in \Q_+.\ a \vartriangleleft_q b$ coincides with the usual uniformly below relation for the metric uniformity.

\begin{definition}
 A \emph{metric locale} is a pre-metric locale $X$ for which $a = \bigvee_{b \vartriangleleft a} b$ for all $a \in \O X$.
 By the above discussion, this is clearly the equivalent to the metric uniformity being admissible.
\end{definition}
Thus, the forgetful functor $\PreMetLoc \to \PreUnifLoc$ restricts to a functor $U$ from the category $\MetLoc$ of metric locales to $\UnifLoc$.

Finally, we note that the completion of a metric locale described in \cite{Henry2016} agrees with the completion of the underlying uniform locale in the sense that the following diagram commutes.
\begin{center}
\begin{tikzpicture}[node distance=2.5cm, auto]
  \node (TL) {$\MetLoc$};
  \node (TR) [right of=TL, xshift=0.6cm] {$\CMetLoc$};
  \node (BL) [below of=TL] {$\UnifLoc$};
  \node (BR) [below of=TR] {$\CUnifLoc$};
  \draw[->] (TL) to [swap] node {$U$} (BL);
  \draw[->] (TR) to node {$U$} (BR);
  \draw[->] (TL) to node {$\Cvar$} (TR);
  \draw[->] (BL) to node [swap] {$\Cvar$} (BR);
\end{tikzpicture}
\end{center}

\subsection{Localic groups}

Another important class of examples of uniform spaces is given by topological groups.
In our setting we will see that every overt localic group has a natural uniform structure.

By \emph{localic group} we simply mean an internal group in $\Loc$.
More explicitly, this means we have a locale $G$ and maps $\epsilon\colon 1 \to G$, $\iota\colon G \to G$ and $\mu\colon G \times G \to G$ satisfying the group axioms.
\begin{remark}
Localic groups are the pointfree analogue of topological groups.
However, even though classically every $T_0$ topological group is Tychonoff and hence sober, not every topological group is the spectrum of a localic group,
since a product of two locales can diverge from that of the corresponding spaces.
The difference between the two theories can be understood in terms of uniform structures: every localic group satisfies a certain completeness condition (as shown in \cite{banaschewski1999completeness} in the classical setting).
\end{remark}

\begin{definition}
 Let $G$ be an overt localic group. The \emph{left uniformity} on $G$ is given by the following set of basic entourages:
 for each open neighbourhood of the identity $u$ we define $L_u$ to be $\{(x,y)\colon G \times G \mid x^{-1}y \in u \}$ in the internal logic.
 
 The \emph{right uniformity} is similar with $R_u = \{(x,y)\colon G \times G \mid xy^{-1} \in u \}$. The \emph{two-sided uniformity} is the join of these in the lattice of uniformities.
 It has a base consisting of the entourages $T_u = \{(x,y)\colon G \times G \mid x^{-1}y \in u \,\land\, xy^{-1} \in u \}$.
\end{definition}
\begin{lemma}
 These all define admissible uniformities on $G$.
\end{lemma}
\begin{proof}
 We will show this for the left uniformity. The proof for the right uniformity is similar and the result for the two-sided uniformity follows from the results for the other two uniformities.
 
 First note that $\{L_u \mid \epsilon^*(u) = \top\}$ is a filter base, since $L_{u \wedge v} = L_u \wedge L_v$.
 We have that each basic entourage $L_u$ is reflexive, since $x^{-1} \cdot x = 1 \in u$.
 For the symmetry condition it is enough to show $L_u^\mathrm{o} = L_{\iota^*(u)}$ and this holds since $x^{-1} y \in u$ if and only if $y^{-1} x = (x^{-1} y)^{-1} \in \iota^*(u)$.
 
 For the transitivity condition, we must find an open neighbourhood of the identity $v$ such that $L_v \circ L_v \le L_u$.
 In the internal logic, we have $1 \cdot 1 = 1 \in u$ and so $\vdash (1,1) \in \mu^*(u)$. Since elements of the form $a \oplus b$ form a base for $\O(G \times G)$, there are $a$ and $b$ such that $(1,1) \in a \oplus b \le \mu^*(u)$.
 Setting $v = a \wedge b$ it follows that $1 \in v$ and $v \oplus v \le \mu^*(u)$.
 Now take $(x,y)$ and $(y,z)$ in $L_v$. Then $x^{-1}y, y^{-1}z \in v \le \mu^*(u)$. Multiplying these we then have $x^{-1}z \in u$ and so $(x,z) \in L_u$.
 Thus, we have shown $L_v \circ L_v \le L_u$ as required.
 
 Finally, we demonstrate admissibility. We first observe that if $\epsilon^*(u) = \top$ and $a \oplus u \le \mu^*(b)$, then $(a \oplus a) \circ L_u \le b \oplus b$.
 To see this it suffices to show $y \in a \land y^{-1}z \in u \vdash_{y,z\colon G} z \in b$ by \cref{lem:uniformly_below_internal_logic}.
 Suppose $y \in a$ and $y^{-1}z \in u$. Then $(y,y^{-1}z) \in a \oplus u \le \mu^*(b)$, so that $z = yy^{-1}z \in b$, as required.
 
 In particular, we have shown that if $\epsilon^*(u) = \top$ and $a \oplus u \le \mu^*(b)$, then $a \vartriangleleft b$.
 So to show admissibility, it suffices to prove that $b \le \bigvee\{ a \in \O G \mid \epsilon^*(u) = \top,\, a \oplus u \le \mu^*(b) \}$.
 But the axiom $x \cdot 1 = x$ means that $(\O G \oplus \epsilon^*)\mu^*$ is the canonical isomorphism $\O G \cong \O G \oplus \Omega$ and so
 \begin{align*}
  b \oplus \top &= (\O G \oplus \epsilon^*)\mu^*(b) \\
    &= (\O G \oplus \epsilon^*)\left(\bigvee \{a \oplus u \mid a \oplus u \le \mu^*(b)\}\right) \\
    &= \bigvee \{a \oplus \epsilon^*(u) \mid a \oplus u \le \mu^*(b)\} \\
    &= \bigvee \{a \oplus \top \mid \epsilon^*(u) = \top,\, a \oplus u \le \mu^*(b)\}
 \end{align*}
 and the desired result follows.
\end{proof}

As before we see this assignment of uniformities is functorial.
\begin{lemma}
 Every homomorphism $f\colon G \to H$ of localic groups is uniform with respect to the left, right or two-sided uniformities.
\end{lemma}
\begin{proof}
 Simply observe that $(f \times f)^*(L_u) = L_{f^*(u)}$ and similarly for $R_u$ and $T_u$.
\end{proof}

Thus, we obtain three functors $\mathcal{L}$, $\mathcal{R}$ and $\mathcal{T}$ from the category $\OLocGrp$ of overt localic groups to $\UnifLoc$ which each commute with the forgetful functors to $\OLoc$.

We can now proceed as in \cite{banaschewski1999completeness} to show that an overt locale localic group is complete with respect to its two-sided uniformity.

\begin{lemma}\label{lem:group_multiplication_lifts}
 The multiplication $\mu\colon G \times G \to G$ on an overt localic group lifts to a map $\widetilde{\mu}\colon \Cvar G \times \Cvar G \to \Cvar G$ on the completion with respect to the left, right or two-sided uniformity.
\end{lemma}
\begin{proof}
 Let us first show this for the completion with respect to the left uniformity.
 By \cref{prop:completion_preserves_products} and \cref{prop:lift_to_completion_via_entourages} it suffices to show that $(\gamma_G \times \gamma_G)_* \mu^*$ sends entourages to reflexive relations.
 Here $(\gamma_G \times \gamma_G)_* \mu^*(c) = \bigvee\{[a \in F] \oplus [b \in F] \mid a \oplus b \le \mu^*(c)\}$
 and so for each open neighbourhood of the identity $u$ we must show \[\bigvee\{[a \in F] \oplus [b \in F] \mid a \oplus b \oplus a \oplus b \le (\mu \times \mu)^*(L_u)\} = 1.\]
 Note that in the internal logic the inequality in above join says that
 \begin{equation}\tag{$\ast$}\label{eq:to_prove_in_showing_mu_lifts}
  x,x' \in a \land y,y' \in b \vdash_{x,y,x',y'\colon G} y^{-1}x^{-1}x'y' \in u.
 \end{equation}
 
 Consider $b$ `sufficiently small' in the sense that $b \oplus b \le L_v$ for some $v$ to be determined
 and set $w = \{ z \colon G \mid \exists \tilde{y}, \tilde{y}' \colon b.\ \tilde{y}^{-1}z\tilde{y}' \in v \}$ in the internal logic. Note that $\vdash 1 \in w$.
 
 Now take $a \in \O G$ such that $a \oplus a \le L_w$. Note that since $L_v$ and $L_w$ are entourages taking the join of $[a \in F] \oplus [b \in F]$ over all such $a$ and $b$ gives \[\bigvee\{[a \in F] \oplus [b \in f] \mid a \oplus a \le L_{w_b},\, b \oplus b \le L_v\} = \bigvee \{1 \oplus [b \in f] \mid b \oplus b \le L_v\} = 1.\]
 This will imply that the desired join is also $1$ once we show that these choices for $a$ and $b$ satisfy condition \eqref{eq:to_prove_in_showing_mu_lifts} above.
 Expanding the definition of $a$ we have $x, x' \in a \vdash_{x,x' \colon G} \exists \tilde{y}, \tilde{y}' \colon b.\ \tilde{y}^{-1}x^{-1}x'\tilde{y}' \in v$,
 which is close to the desired condition. We must just ensure that each $y$ and $y'$ in $b$ are close to the given $\tilde{y}$ and $\tilde{y}'$.
 
 Indeed, if $x,x' \in a$ and $y,y' \in b$ then we know there exist $\tilde{y}, \tilde{y}' \in b$ such that $(x\tilde{y}, x'\tilde{y}') \in L_v$.
 But also, $y, \tilde{y} \in b$ implies $(y,\tilde{y}) \in L_v$, which in turn gives $(xy,x\tilde{y}) \in L_v$.
 Similarly, $(x'\tilde{y}',x'y') \in L_v$. Combining these we find $(xy,x'y') \in L_v \circ L_v \circ L_v$. Choosing $L_v$ so that $L_v \circ L_v \circ L_v \le L_u$ then yields the desired result.
 
 Thus, we have shown $(\gamma_G \times \gamma_G)_* \mu^*$ sends the entourages $L_u$ to reflexive relations.
 Entirely dually, the entourages $R_u$ for the right uniformity are also sent to reflexive relations.
 
 To prove a similar result for the basic two-sided entourages $T_u = L_u \wedge R_u$, we consider both $b \oplus b \le T_v$ and $a \oplus a \le T_{w_b}$ as above, as well as $a' \oplus a' \le T_{v'}$ and $b' \oplus b' \le T_{w'_{\smash{a'}}}$ as for the right uniformity. Then the join of $[a \wedge a' \in F] \oplus [b \wedge b' \in F]$ over all such $a$, $a'$, $b$ and $b'$ is again a cover, since it is the meet of two covers.
 Finally, for $x, x' \in a \wedge a'$ and $y,y' \in b \wedge b'$, we find $(xy,x'y') \in L_u$ as above and dually $(xy,x'y') \in R_u$. Thus, $(xy, x'y') \in L_u \wedge R_u = T_u$.
 Thus, $\bigvee\{[a \in F] \oplus [b \in F] \mid a \oplus b \oplus a \oplus b \le (\mu \times \mu)^*(T_u)\}$ is also $1$ and $(\gamma_G \times \gamma_G)_* \mu^*$ sends two-sided entourages to reflexive relations.
 Hence $\mu$ lifts to the completion with respect to the two-sided uniformity.
\end{proof}

\begin{theorem}
 Every overt localic group $G$ is complete with respect to its two-sided uniformity.
\end{theorem}
\begin{proof}
 By the previous lemma the multiplication $G$ lifts to the completion $\Cvar G$. Moreover, the inversion map $\iota$ simply exchanges left and right entourages, so it uniformly continuous with respect to the two-sided uniformity and hence clearly also lifts to the completion. Of course, the unit lifts too. Since $G^n$ is a strongly dense sublocale of $(\Cvar G)^n$, the group axioms hold on a strongly dense sublocale of $\Cvar G$.
 But since $\Cvar G$ is uniformisable, it is weakly Hausdorff and so the sublocale on which the axioms hold is weakly closed. Thus, they must hold everywhere.
 
 It follows that $\Cvar G$ is a localic group and $G$ is a strongly dense localic subgroup. But it is well-known that an overt subgroup of a localic group is weakly closed (see \cite{johnstone1989constructive} for a constructive proof). Thus, $G \cong \Cvar G$ and $G$ is complete.
\end{proof}

\subsection{Real numbers and the completion of rings}

Another important application of completions is the construction of the localic ring of reals as the completion of the ring of rational numbers.
The rationals are best viewed as a discrete locale $\Q$. This has a pre-uniform structure which can be given by the usual Euclidean metric or directly from a base of entourages of the form $E_q = \{(x, x') \in \Q \times \Q \mid \abs{x-x'} < q\}$ for $q \in \Q_+$.

We can use \cref{prop:completion_from_a_base} to obtain a presentation for the completion of $\Q$. An appropriate set $\mathcal{B}$ of basic uniform covers consists of $C_q = {\downarrow}\{(p,p+q) \subseteq \Q \mid p \in \Q\}$ for $q \in \Q_+$ and where $(r,s)$ denotes the interval of rationals lying strictly between $r$ and $s$ (for $r < s$).
The resulting presentation of $\Cvar\Upsilon \Q$ has generators $(\!(r,s)\!)$ for $r < s \in \Q$ and the relations:
\begin{itemize}
  \item $(\!(r,s)\!) \wedge (\!(r',s')\!) = (\!(r \vee r', s \wedge s')\!)$ for $r,r' < s,s'$,
  \item $(\!(r,s)\!) \wedge (\!(r',s')\!) = 0$ otherwise,
  \item $\bigvee_{p \in \Q} (\!(p, p+q)\!) = 1$ for $q > 0$,
  \item $(\!(r,s)\!) \le \bigvee_{r < r' < s' < s} (\!(r',s')\!)$.
 \end{itemize}
It is easy to see this is isomorphic to the usual frame of reals $\O \R$ as presented in terms of Dedekind cuts (with generators $\ell_q$ and $u_q$ meaning $q$ is in the lower or upper cut respectively) via frame homomorphisms sending $(\!(r,s)\!)$ to $\ell_r \wedge u_s$ in the one direction and sending $\ell_r$ to $\bigvee_{s > r} (\!(r,s)\!)$ and $u_s$ to $\bigvee_{r < s} (\!(r,s)\!)$ in the other direction. A proof of a similar equivalence can be found in \cite[Chapter XIV, Proposition 1.2.1]{picado2012book}.

Now let us discuss the lifting of the ring operations to $\R$. We note that the uniformity on $\Q$ is translation invariant.
\begin{definition}
 A uniformity on a localic abelian group $G$ is \emph{translation invariant} if it has a base of entourages $E$ such that $(x,y) \in E \vdash_{x,y,z\colon G} (x+z,y+z) \in E$.
\end{definition}
For translation-invariant entourage $E$ we always have $(x,y) \in E \dashv\vdash_{x,y\colon G} (x-y,0) \in E$ (by taking $z = \pm y$). Now since $u \coloneqq \{z\colon G \mid (z,0) \in E\}$ is an open neighbourhood of the identity, each basic entourage in a translation-invariant uniformity is of the form $T_u$ (the three group uniformities coincide for abelian $G$). In this way, a translation-invariant uniformity can be specified by giving a collection of open neighbourhoods of the identity.

The operations of a localic abelian group always respect translation-invariant pre-uniform structures.

\begin{lemma}
 Let $\E$ be a translation-invariant uniformity on a localic abelian group $G$. Then the operations of $G$ are uniformly continuous with respect to $\E$.
\end{lemma}
\begin{proof}
 Inversion is uniformly continuous by commutativity and it is trivial that the unit is uniformly continuous. For the addition we consider an entourage $E \in \E$ and take a symmetric entourage $F \in \E$ such that $F \circ F \le E$. To show addition is uniformly continuous it is enough to show that $(x,x'), (y,y') \in F \vdash_{x,x',y,y'\colon G} (x+y, x'+y') \in E$. But if $(x,x'), (y,y') \in F$ then $(x+y,x'+y) \in F$ and $(x'+y,x'+y') \in F$ by translation invariance. Thus, $(x+y,x'+y') \in F \circ F \le E$ as required.
\end{proof}
\begin{remark}
 This yields an analogue of \cref{lem:group_multiplication_lifts} restricted to abelian groups, but for general translation-invariant uniformities.
 In fact, the restriction to the abelian case is not necessary.
 If we define a notion of left- or right-translation-invariant uniformity on a non-abelian group $G$, then a simple modification of the proof of \cref{lem:group_multiplication_lifts} shows that the multiplication on $G$ lifts to the `pre-uniform completion' of $G$ with respect to any such uniformity --- and for any uniformity with a base of entourages of the form $T_u$.
\end{remark}
\begin{corollary}
 The (pre-uniform) completion with respect to a translation-invariant uniformity preserves the abelian group structure of $G$.
 Moreover, this gives a reflection from the category of translation-invariant pre-uniform localic abelian groups onto the category of localic abelian groups (with their canonical group uniformities).
\end{corollary}

Thus, in particular the additive structure on $\Q$ extends to $\R$. That the entire ring structure extends to $\R$ is a consequence of the following general result, which is a pointfree analogue of \cite[III, \S 6.5, Proposition 6]{bourbaki2013general}.
\begin{proposition}\label{prop:localic_ring_completion}
 Let $R$ be a localic ring equipped with a pre-uniform structure that is translation invariant with respect to its additive group.
 Furthermore, suppose the multiplication $\mu_\times\colon R \times R \to R$ is `continuous at $0$' with respect to the pre-uniform structure in the sense that whenever $T_u$ is an entourage, there is an entourage $T_{v}$ such that $v \oplus v \le \mu_\times^*(u)$.
 Then $\Cvar\Upsilon R$ has the structure of a localic ring and the map $\gamma_{\Upsilon R} \upsilon_R\colon R \to \Cvar\Upsilon R$ is a localic ring homomorphism.
\end{proposition}
\begin{proof}
 It suffices to show that the multiplication lifts to the completion.
 Similarly to before, by \cref{prop:lift_to_completion_via_entourages} it is enough to show \[\bigvee\{[\upsilon_*(a) \in F] \oplus [\upsilon_*(b) \in F] \mid a \oplus b \oplus a \oplus b \le (\mu_\times \times \mu_\times)^*(T_u)\} = 1\]
 for each basic entourage $T_u$.
 Note that the inequality  can be expressed in the internal logic as
 $x,x' \in a \land y,y' \in b \vdash_{x,y,x',y' \colon G} xy-x'y' \in u$. Now consider the following sequence of equalities.
 \begin{align*}
  xy-x'y' &= xy - xy' + xy' - x'y' \\
          &= x(y-y') + (x-x')y' \\
          &= (\tilde{x} + x - \tilde{x})(y-y') + (x-x')(\tilde{y} + y' - \tilde{y}) \\
          &= \tilde{x}(y-y') + (x - \tilde{x})(y-y') + (x-x')\tilde{y} + (x-x')(y' - \tilde{y})
 \end{align*}
 By the uniform continuity of addition there is a $u'$ such that $T_{u'}$ is an entourage and whenever each of these terms lies in $u'$ their sum $xy-x'y'$ lies in $u$.
 
 Consider $a \in \O R$ such that $a \oplus a \le T_v$ for some entourage $T_v$ to be determined later.
 Note that $\bigvee\{[\upsilon_*(a) \in F] \mid a \oplus a \le T_v\} = 1$ since $T_v$ is an entourage.
 We set $w = \{z\colon R \mid z \in v \land \exists \tilde{x}\colon a.\ \tilde{x}z \in u'\}$ and note that $\vdash 1 \in w$. Now take $b \oplus b \le T_w$. As before we find that the join of $[\upsilon_*(a) \in F] \oplus [\upsilon_*(b) \in F]$ over all such $a$ and $b$ is $1$.
 In the internal logic, if $x,x' \in a$ and $y,y' \in b$ we have $x-x', y-y' \in v$ and some $\tilde{x} \in a$ such that $\tilde{x}(y-y') \in u'$.
 
 We also wish to show $(x-\tilde{x})(y-y') \in u'$.
 We know $x - \tilde{x} \in v$ and $y-y' \in v$. Now we use `continuity at 0' to choose $v$ such that $v \oplus v \le \mu_\times^*(u')$, which gives the desired result.
 
 Similarly, a dual choice of $a$ and $b$ gives that the other two terms lie in $u'$. Taking the meet of these two choices of $a$ and $b$ then means all of the terms lie in $u'$ and hence $xy - x'y' \in u$.
 Thus, the result follows as in \cref{lem:group_multiplication_lifts}.
\end{proof}
\begin{corollary}
 The ring structure on $\Q$ lifts to a ring structure on $\R$.
\end{corollary}
\begin{proof}
 By the above it suffices to show the multiplication on $\Q$ is `continuous at 0'. This amounts to showing for each $\epsilon \in \Q_+$ there is a $\delta \in \Q_+$ such that if $\abs{x}, \abs{y} < \delta$ then $\abs{xy} < \epsilon$.
 We can now simply take $\delta = \min(\epsilon, 1)$ so that $\abs{xy} = \abs{x}\abs{y} < \delta^2 \le \epsilon \cdot 1 = \epsilon$.
\end{proof}
\begin{remark}
 We can also show in an entirely analogous way (using the $p$-adic absolute value) that the $p$-adic uniformity on $\Q$ is translation invariant and continuous at $0$ and hence there is an induced localic ring structure on the resulting locale of $p$-adic numbers.
\end{remark}

\hypersetup{bookmarksdepth=-1}
\subsection*{Acknowledgements}

I acknowledge financial support from the Centre for Mathematics of the University of Coimbra (UIDB/00324/2020, funded by the Portuguese Government through FCT/MCTES).

I would also like to thank José Siqueira for finding an online copy of the thesis of Fox \cite{fox2005thesis} for me. In case anyone else is also looking for it, it can be found on the University of Manchester Library website at \url{https://uomlibrary.access.preservica.com/uncategorized/IO_7a11b53c-c8fb-4b0d-9d3d-ed91949d7ce0}.

\bibliographystyle{abbrv}
\bibliography{references}

\end{document}